\documentclass[10pt]{article}
\usepackage{geometry}                
\usepackage{graphicx}
\usepackage[utf8]{inputenc} 
\usepackage{textcomp} 

\usepackage{flafter}  

\usepackage{amsmath,amssymb}  
\usepackage{bm}  
\usepackage{amsthm}

\usepackage{memhfixc}  
\usepackage{pdfsync}  

\newtheorem{definition}{Definition}[section]
\newtheorem{theorem}{Theorem}[section]
\newtheorem{proposition}{Proposition}[section]
\newtheorem{corollary}{Corollary}[section]
\newtheorem{remark}{Remark}[section]
\newtheorem{lemma}{Lemma}[section]
\numberwithin{equation}{section}

\title{Insertion and Lie Bracket Concerning Finite Sets}
\author{Zhou Mai \footnote{address:Colleague of Mathematical Science, Nankai University, Weijin Road, Tianjin City, Republic China;
            email address: zhoumai@nankai.edu.cn}}

\begin{document}

\maketitle

\begin{abstract}
In this article we discuss the operations of partitions
(sequence of disjoint finite subsets) which are
quotient, insertion, composition and Lie bracket.
Moreover, we discuss applications of those operations for
Feymman diagrams and Kontesvich's graphs.

$\bf{Keywords:}$ partition, quotient, insertion,
Lie bracket, Feymman diagrams, admissible graphs.
\end{abstract}

\tableofcontents

\section{Introduction}

In this article we discuss some operations of 
partitions, where the partition means the
sequence of disjoint finite subsets.
The operations what we focus on include quotient,
insertion, composition and Lie bracket.
All of operation is generalization of ones
concerning Feymman diagrams (see A. Connes and D. Kreimer
\cite{1, 2}, A. Connes and M. Marcolli \cite{3}, 
D. Kreimer \cite{6})
and Kontesvich's graphs (see L. M. Ionescu \cite{4}, M. Kontesvich \cite{7}).
The construction in this article is suitable for the cases of
Feymman diagrams and Kontesvich's graphs,
actually, if we consider some structure maps additionally,
the operations concerning Feymman diagrams
and Kontesvich's graphs can be reduced to
our construction. Here the construction of
quotient follows the ideas in  Zhou Mai \cite{9},
but, some modification occurs such that it is
more suitable for the cases of Feymman diagrams
and Kontesvich's graphs. Our construction 
is suitable for ordinary graphs in the 
sense of graphic theory as well.
Somehow, a ordinary graph can be regarded
as a Feymman diagram without external lines,
but, the case of subgraph is different.
It seems that three types of above graphs can be dealt
with in an uniform way.

This paper is organized as follows. In section 2
we discuss the quotient and insertion of partitions
in details. Based on the quotient we construct
the coproduct which will result in a hopf algebra,
but we do not discuss this issue more.
In section 3 we construct the composition
and Lie bracket. Here two types of composition
are considered, both of them will result in well defined
Lie bracket. Finally, in section 4 we discuss
the cases of Feymman diagrams, Kontesvich's
graphs and ordinary graphs starting from
our construction.

\section{Quotient and insertion of the partitions}

\subsection{Notations concerning the partitions}

\paragraph{Partitions:}

Firstly, we introduce some notations which will
be useful for discussion below.
\begin{itemize}

\item 
For a finite set $A$, we denote the power
set of $A$ by $\mathcal{P}(A)$.
The reversion map
$\mathcal{R}:\mathcal{P}(\mathcal{P}(A))
\rightarrow \mathcal{P}(A)$ is defined
to be: 
$$
\mathcal{R}(\{I_{i}\})=
\bigcup_{i}I_{i},\,\{I_{i}\}\in 
\mathcal{P}(\mathcal{P}(A)).
$$
Let $\mathcal{P}_{dis}^{2}(A)$
denote a subset of $\mathcal{P}(\mathcal{P}(A))$,
$$
\mathcal{P}_{dis}^{2}(A)=
\{\{I_{i}\}\in 
\mathcal{P}(\mathcal{P}(A))|
\{I_{i}\}\in\mathbf{Part}(\mathcal{R}(\{I_{i}\}))\}.
$$
We call the element $\{I_{i}\}\in\mathcal{P}_{dis}^{2}(A)$
the partition in $A$. If we ignore the order of $\{I_{i}\}$,
then we identify $\{I_{i}\}$ with $\{I_{\sigma(i)}\}$,
where $\sigma\in \mathbb{S}_{m}$ and $m=|\{I_{i}\}|$
($|B|$ denotes the number of the elements in a finite
set $B$).

A partition $\{I_i{}\}_{i=1}^{m}\in\mathcal{P}_{dis}^{2}(A)$
can be decribed by a function $f$ 
from $I=\mathcal{R}(\{I_{i}\})$
to $\mathbb{N}$ with $|Im(f)|=m$.
Precisely, let $Im(f)=\{i_{1},\cdots,i_{m}\}$
($0<i_{1}<\cdots<i_{m}$), then 
$\{f^{-1}(i_{k})\}_{k=1}^{m}$ is a partition
in $A$. If $f^{-1}(i_{k})=I_{k}$, then the function $f$
defines the partition $\{I_{i}\}_{i=1}^{m}$.
For a permutation
$\sigma:\{i_{1},\cdots,i_{m}\}\rightarrow
\{i_{\sigma(1)},\cdots,i_{\sigma(m)}\}$,
it is obvious that $\sigma\circ f$ defines
a same partition. On the other hand, let
$\tau$ be a map from $\{i_{1},\cdots,i_{m}\}$
to $\underline{m}=\{1,\cdots,m\}$,
$\tau(i_{k})=k$ ($k=1,\cdots,m$),
then $\tau\circ f$ defines the same partition
also. Without loss of generality, we can always
assume $Im(f)=\underline{m}$. We call the function
$f:I\rightarrow\underline{m}$ satisfying
$f^{-1}(i)=I_{i}$ ($i=1,\cdots,m$) the defining
function of the partition $\{I_{i}\}_{i=1}^{m}$
denoted by $f_{\{I_{i}\}}$. 

\item 
Let $\{I_{i}\}_{i=1}^{m},\{J_{j}\}_{j=1}^{n}\in\mathcal{P}_{dis}^{2}(A)$
with defining functions $f_{\{I_{i}\}},\,f_{\{J_{j}\}}$,
$\mathcal{R}(\{I_{i}\})\cap\mathcal{R}(\{J_{j}\})\\=\emptyset$,
Then $\{I_{1},\cdots,I_{m},J_{1},\cdots,J_{n}\}\in\mathcal{P}_{dis}^{2}(A)$,
the new partition as above is denoted by $\{I_{i}\}\cup\{J_{j}\}$.
Let $f_{\{I_{i}\}\cup\{J_{j}\}}$ be the defining
function of $\{I_{i}\}\cup\{J_{j}\}$,
then, $f_{\{I_{i}\}\cup\{J_{j}\}}|_{I}=f_{\{I_{i}\}}$,
$f_{\{I_{i}\}\cup\{J_{j}\}}|_{J}=\iota\circ f_{\{J_{j}\}}$,
where $I=\mathcal{R}(\{I_{i}\})$,
$J=\mathcal{R}(\{J_{j}\})$ and
$\iota(j)=j+m$ ($j=1,\cdots,n$).
We denote $f_{\{I_{i}\}\cup\{J_{j}\}}$
by $f_{\{I_{i}\}}\cup f_{\{J_{j}\}}$ also.
It is obvious that $f_{\{J_{j}\}}\cup f_{\{I_{i}\}}$
defines same partition, therefore, we will
identify $f_{\{J_{j}\}}\cup f_{\{I_{i}\}}$
with $f_{\{I_{i}\}}\cup f_{\{J_{j}\}}$.

\item 
Let $\{I_{i}\},\{J_{j}\}\in\mathcal{P}_{dis}^{2}(A)$,
we say $\{J_{j}\}\subset\{I_{i}\}$, if for any $j$,
there is an $i$, such that $J_{j}\subset I_{i}$.
An important situation is $\{J_{j}\}=\{I_{i}\cap B\}$,
where $B\subset\mathcal{R}(\{I_{i}\})$.
It is obvious that $\mathcal{R}(\{J_{j}\})=B$.
In this situation we call $\{J_{j}\}$ 
the restriction of $\{I_{i}\}$ on $B$,
denoted by $\{J_{j}\}=\{I_{i}\}_{(B)}$.
The defining function of $\{J_{j}\}$ can be
determined as follows. Let $\{i_{1},\cdots,i_{n}\}
=\{i|1\leq i\leq m, I_{i}\cap B\not=\emptyset\}$,
where $i_{1}<\cdots<i_{n}$ and $n=|\{J_{j}\}|$.
If we take $J_{j}=I_{i_{j}}\cap B$, then, 
$f_{\{J_{j}\}}$ should be such a function,
$f_{\{J_{j}\}}:B\rightarrow\underline{n}$,

$$
(f_{\{J_{j}\}})^{-1}(j)=
(f_{\{I_{i}\}})^{-1}(i_{j})\cap B=J_{j}.
$$
Thus, $f_{\{I_{i}\}}|_{B}=\iota\circ f_{\{I_{i}\}}$,
where $\iota:\underline{n}\rightarrow
\{i_{1},\cdots,i_{n}\}$, $\iota(j)=i_{j}$
($j=1,\cdots,n$). If we do not distinguish
$f_{\{J_{j}\}}$ and $\iota\circ f_{\{I_{i}\}}$, 
then we have $f_{\{I_{i}\}}|_{B}=f_{\{I_{i}\}}$.
In fact the functions to define a partition is
an equivalent class under a natural equivalent
relation, above formula is exactly valid
in the sense of the equivalent relation.

If $D\subset B\subset \mathcal{R}(\{I_{i}\})$,
$\{J_{j}\}=\{I_{i}\}_{(B)}$, we have
$\{J_{j}\}_{(D)}=\{I_{i}\}_{(D)}$. 

\end{itemize}

\paragraph{Map-union of the partitions}

Let $\{I_{i}\}\in\mathcal{P}_{dis}^{2}(A)$,
$B\subset A$ be a subset, 
$B\cap\mathcal{R}(\{I_{i}\})=\emptyset$,
$f:B\to\mathcal{R}(\{I_{i}\})$ ba a map, we define a new 
partition $\{f^{-1}(I_{i})\cup I_{i}\}$ which is called
the map-union of $\{I_{i}\}$ and $B$ by $f$,
denoted by $B\sqcup_{f}\{I_{i}\}$.

Now we discuss some properties of the map-union.

\begin{itemize}
\item
Let $\{I_{i}\}_{i=1}^{m}\in\mathcal{P}_{dis}^{2}(A)$
with defining function $f_{\{I_{i}\}}$,
$B\subset A$ be a subset, 
$B\cap I=\emptyset$ ($I=\mathcal{R}(\{I_{i}\})$),
for two maps $f:B\to I$, $g:B\to I$, then

$$
B\sqcup_{f}\{I_{i}\}=B\sqcup_{g}\{I_{i}\}\,\Leftrightarrow
f^{-1}(I_{i})=g^{-1}(I_{i})\,(1\leq i\leq m)\Leftrightarrow
f_{\{I_{i}\}}\circ f=f_{\{I_{i}\}}\circ g.
$$
Thus, the map-union $B\sqcup_{f}\{I_{i}\}$ 
depends only on $f_{\{I_{i}\}}\circ f$.
With the help of defining function $f_{\{I_{i}\}}$,
there is a one-one correspondence
between $I_{i}$ and its index $i$, the
map $f_{\{I_{i}\}}\circ f$ can be regarded
as a map $\iota:B\rightarrow\{I_{i}\}$,
and the map-union $B\sqcup_{f}\{I_{i}\}$
can be expressed as $\{\iota^{-1}(I_{i})\cup I_{i}\}$
denoted by $B\sqcup_{\iota}\{I_{i}\}$.

\item For two partitions
$\{I_{i}\},\{J_{j}\}\in\mathcal{P}_{dis}^{2}(A)$,
and two subsets $B,C\subset A$, which satisfy
$I\cap J=\emptyset$, $B\cap C=\emptyset$,
$(B\cup C)\cap(I\cup J)=\emptyset$, where
$I=\mathcal{R}(\{I_{i}\})$,
$J=\mathcal{R}(\{J_{j}\})$, it is obvious that

$$
(B\cup C)\sqcup_{f\cup g}(\{I_{i}\}\cup\{J_{j}\})=
(B\sqcup_{f}\{I_{i}\})\cup(C\sqcup_{g}\{J_{j}\}),
$$
where $f:B\to I$, $g:C\to J$,
$f\cup g:B\cup C\to I\cup J$,
$(f\cup g)|_{B}=f$, $(f\cup g)|_{C}=g$.

\item Let $\{I_{i}\}\in\mathcal{P}_{dis}^{2}(A)$,
$B,C\subset A$, which satisfy
$B\cap I=\emptyset$, $C\subset I$, where 
$I=\mathcal{R}(\{I_{i}\})$, then

$$
B\sqcup_{f}\{I_{i}\}=(B_{1}\sqcup_{f|_{B_{1}}}\{I_{i}\}_{(I\cap C)})
\cup(B_{2}\sqcup_{f|_{B_{2}}}\{I_{i}\}_{(I\setminus C)}),
$$
where $f:B\to I$, 
$B_{1}=f^{-1}(I\cap C)$,
$B_{2}=f^{-1}(I\setminus C)$.

\item Let $B,C\subset A$,
$\{I_{i}\}\in\mathcal{P}_{dis}^{2}(A)$,
$B\cap C=\emptyset$, $I\cap(B\cup C)=\emptyset$,
$I=\mathcal{R}(\{I_{i}\})$, then we have
$$
(B\cup C)\sqcup_{f\cup g}\{I_{i}\}
=B\sqcup_{f}(C\sqcup_{g}\{I_{i}\})
=C\sqcup_{g}(B\sqcup_{f}\{I_{i}\}),
$$
where
$f:B\to I$, $g:C\to I$,
$f\cup g:B\cup C\to I$,
$(f\cup g)|_{B}=f$, $(f\cup g)|_{C}=g$.

\end{itemize}

\subsection{Quotient of the partitions}

Now we define the quotient between two partitions.

\begin{definition}
Let $\{I_{i}\}\in\mathcal{P}_{dis}^{2}(A)$,
$B\subset\mathcal{R}(\{I_{i}\})$, we define the quotient
of $\{I_{i}\}$ by $\{I_{i}\}_{(B)}$, 
denoted by $\{I_{i}\}\diagup\{I_{i}\}_{(B)}$,
to be
\begin{equation}
\{I_{i}\}\diagup\{I_{i}\}_{(B)}=\{I_{i}\}_{I_{i}\cap B=\emptyset}
\cup\{\mathcal{R}(\{I_{i}\}_{I_{i}\cap B\not=\emptyset})
\setminus B\}.
\end{equation}
We call $\{\mathcal{R}(\{I_{i}\}_{I_{i}\cap B\not=\emptyset})
\setminus B\}$ the ideal part of the quotient 
$\{I_{i}\}\diagup\{I_{i}\}_{(B)}$.
\end{definition}

\begin{remark}
\begin{itemize}
\item It is convenience to denote 
$\mathcal{R}(\{I_{i}\}_{I_{i}\cap B\not=\emptyset})$
by $\mathcal{R}_{\{I_{i}\},B}$, or, 
$\mathcal{R}_{B}$ for short sometime.
From definition 2.1 we can easily see that

$$
\mathcal{R}(\{I_{i}\})=\mathcal{R}(\{I_{i}\}\diagup\{I_{i}\}_{(B)})
\cup B, \,\mathcal{R}(\{I_{i}\}\diagup\{I_{i}\}_{(B)})
\cap B=\emptyset.
$$
Therefore

$$
\mathcal{R}(\{I_{i}\}\diagup\{I_{i}\}_{(B)})=
\mathcal{R}(\{I_{i}\})\setminus B.
$$
\item Particularly, when $\mathcal{R}_{B}=B$, 
which means that if $I_{i}\cap B\not=\emptyset$
we have $I_{i}\subset B$, or, 
$\{I_{i}\}_{(B)}=\{I_{i}\}_{I_{i}\cap B\not=\emptyset}$
which is a subset of $\{I_{i}\}$, then
$\{I_{i}\}\diagup\{I_{i}\}_{(B)}
=(\{I_{i}\}\setminus\{I_{i}\}_{(B)})\cup\{\emptyset\}$.
We will identify $(\{I_{i}\}\setminus\{I_{i}\}_{(B)})\cup\{\emptyset\}$
with $\{I_{i}\}\setminus\{I_{i}\}_{(B)}$.
In this special situation the ideal part of quotient
$\{I_{i}\}\diagup\{I_{i}\}_{(B)}$ is $\{\emptyset\}$,
we call this special situation the travail quotient.
Furthermore, $\{I_{i}\}\diagup\{I_{i}\}_{(I)}
=\{I_{i}\}\setminus\{I_{i}\}=\emptyset$.

\item Noting that $\{I_{i}\}=\{I_{i}\}_{I_{i}\cap B=\emptyset}\cup
\{I_{i}\}_{I_{i}\cap B\not=\emptyset}$,
and $\{I_{i}\}_{I_{i}\cap B\not=\emptyset}
\diagup \{I_{i}\}_{(B)}=\{\emptyset\}\cup
\{\mathcal{R}_{B}\setminus B\}$,
thus we can rewrite the formula
(2.1) in the following form:

$$
\{I_{i}\}\diagup\{I_{i}\}_{(B)}=
\{I_{i}\}_{I_{i}\cap B=\emptyset}\cup
(\{I_{i}\}_{I_{i}\cap B\not=\emptyset}
\diagup \{I_{i}\}_{(B)}).
$$
\item Let $J\subset A$, $J\cap I=\emptyset$,
$B\subset I$,
$I=\mathcal{R}(\{I_{i}\})$, then the
quotient $\{I_{i}\}\diagup\{I_{i}\}_{(B)}$
induces the quotient of map-union $J\sqcup_{f}\{I_{i}\}$, where
$f:J\to \{I_{i}\}$ is a map. Actually, let $p_{B}$
be the projection from $\mathcal{R}(\{I_{i}\})$ to
$\{I_{i}\}\diagup\{I_{i}\}_{(B)}$,
$p_{B}(I_{i})=I_{i}$ for $I_{i}\cap B=\emptyset$,
$p_{B}(I_{i})=\mathcal{R}_{B}\setminus B$
for $I_{i}\cap B\not=\emptyset$,
it is easy to check that

$$
(J\sqcup_{f}\{I_{i}\})\diagup\{I_{i}\}_{(B)}=
J\sqcup_{p_{B}\circ f}(\{I_{i}\}\diagup\{I_{i}\}_{(B)}).
$$
\end{itemize}
\end{remark}

We will take a look at the properties
of the quotient. Here we focus on the case of

$$
(\{I_{i}\}\diagup\{I_{i}\}_{(B)})\diagup
(\{I_{i}\}\diagup\{I_{i}\}_{(B)})_{(C)},
$$
where $B,C\subset I=\mathcal{R}(\{I_{i}\})$,
$B\cap C=\emptyset$. The key point is that
when $\mathcal{R}_{B}\cap\mathcal{R}_{C}=\emptyset$,
we have the following decomposition 

$$
\{I_{i}\}_{I_{i}\cap(B\cup C)\not=\emptyset}=
\{I_{i}\}_{I_{i}\cap B\not=\emptyset}\cup
\{I_{i}\}_{I_{I}\cap C\not=\emptyset}.
$$
Otherwise, above decomposition
is not valid.

\begin{lemma}
Let $\{I_{i}\}\in\mathcal{P}_{dis}^{2}(A)$,
$B,C\subset \mathcal{R}(\{I_{i}\})$,
$B\cap C=\emptyset$,
then following formulas are valid.

\begin{itemize}
\item

\begin{equation}
(\{I_{i}\}\diagup\{I_{i}\}_{(B)})_{(C)}=
\{I_{i}\}_{(B\cup C)}\diagup\{I_{i}\}_{(B)}.
\end{equation}

\item

\begin{equation}
\begin{array}{c}
(\{I_{i}\}\diagup\{I_{i}\}_{(B)})\diagup
(\{I_{i}\}\diagup\{I_{i}\}_{(B)})_{(C)} \\
=(\{I_{i}\}\diagup\{I_{i}\}_{(B)})\diagup
(\{I_{i}\}_{(B\cup C)}\diagup\{I_{i}\}_{(B)}).
\end{array}
\end{equation}

\item If $\mathcal{R}_{B}\cap
\mathcal{R}_{C}=\emptyset$, then

\begin{equation}
\{I_{i}\}_{(B\cup C)}\diagup\{I_{i}\}_{(B)}
=\{I_{i}\}_{(C)},
\end{equation}
and

$$
(\{I_{i}\}\diagup\{I_{i}\}_{(B)})_{(C)}=\{I_{i}\}_{(C)}.
$$
\end{itemize}
\end{lemma}

\begin{proof}
At first, in order to prove the formula (2.2), 
we calculate the left side of the formula.
From definition 2.1 we have
$$
\begin{array}{c}
(\{I_{i}\}\diagup\{I_{i}\}_{(B)})_{(C)} \\
=(\{I_{i}\}_{I_{i}\cap B=\emptyset}\cup
\{\mathcal{R}_{B}\setminus B\})_{(C)} \\
=\{I_{i}\cap C\}_{I_{i}\cap B=\emptyset}\cup
\{\mathcal{R}_{B}\cap C\}.
\end{array}
$$
On the other hand, for the right side of the formula (2.2)
we have
$$
\begin{array}{c}
\{I_{i}\}_{(B\cup C)}\diagup\{I_{i}\}_{(B)} \\
=\{I_{i}\cap(B\cup C)\}_{I_{i}\cap B=\emptyset}\cup
\{[\mathcal{R}_{B}\cap(B\cup C)]\setminus B\} \\
=\{I_{i}\cap C\}_{I_{i}\cap B=\emptyset}\cup
\{\mathcal{R}_{B}\cap C\}.
\end{array}
$$
Therefore, the formula (2.2) is valid.
The formula (2.3) is the corollary of
the formula (2.2).

We turn to prove the formula (2.4).
Because $\mathcal{R}_{B}\cap
\mathcal{R}_{C}=\emptyset$, we know that

$$
I_{i}\cap B\not=\emptyset\Longleftrightarrow
I_{i}\cap C=\emptyset\,
(or\,I_{i}\cap C\not=\emptyset\Longleftrightarrow
I_{i}\cap B=\emptyset),\,
if\, I_{i}\cap(B\cup C)\not=\emptyset.
$$
Thus

$$
\{I_{i}\}_{(B\cup C)}=
\{I_{i}\}_{(B)}\cup\{I_{i}\}_{(C)},
$$
then, in this situation the quotient is travial,
that is

$$
\{I_{i}\}_{(B\cup C)}\diagup(\{I_{i}\}_{(B\cup C)})_{(B)}
=\{I_{i}\}_{(B\cup C)}\diagup\{I_{i}\}_{(B)}
=\{I_{i}\}_{(B\cup C)}\setminus\{I_{i}\}_{(B)}
=\{I_{i}\}_{(C)}.
$$
Combining the formulas (2.2), (2.4) we can get
tha last formula in lemma 2.1.

\end{proof}

\begin{proposition}
Let $\{I_{i}\}\in\mathcal{P}_{dis}^{2}(A)$,
$B,C\subset \mathcal{R}(\{I_{i}\})$,
$B\cap C=\emptyset$,
then following formulas are valid.
\begin{itemize}
\item When 
$\mathcal{R}_{B}\cap \mathcal{R}_{C}\not=\emptyset$,

\begin{equation}
(\{I_{i}\}\diagup\{I_{i}\}_{(B)})\diagup(\{I_{i}\}_{(B\cup C)}\diagup\{I_{i}\}_{(B)})
=\{I_{i}\}\diagup\{I_{i}\}_{(B\cup C)}.
\end{equation}

\item When 
$\mathcal{R}_{B}\cap \mathcal{R}_{C}=\emptyset$, 
we have

\begin{equation}
(\{I_{i}\}\diagup\{I_{i}\}_{(B)})\diagup(\{I_{i}\}_{(B\cup C)}\diagup\{I_{i}\}_{(B)})
=(\{I_{i}\}\diagup\{I_{i}\}_{(B)})\diagup\{I_{i}\}_{(C)}.
\end{equation}
\end{itemize}
\end{proposition}

\begin{proof}
From definition of quotient, we know that

$$
\{I_{i}\}\diagup\{I_{i}\}_{(B)} 
=\{I_{i}\}_{I_{i}\cap B=\emptyset}\cup
\{\mathcal{R}_{B}\setminus B\},
$$
and

$$
\{I_{i}\}\diagup\{I_{i}\}_{(B\cup C)}=
\{I_{i}\}_{I_{i}\cap(B\cup C)=\emptyset}\cup
\{\mathcal{R}_{B\cup C}\setminus(B\cup C)\}.
$$

The discussions will be divided into 
two situations.

$\mathbf{Case\,\,of\,\,\mathcal{R}_{B}
\cap \mathcal{R}_{C}\not=\emptyset:}$

By the definition of the quotient
we have

$$
\begin{array}{c}
(\{I_{i}\}\diagup\{I_{i}\}_{(B)})\diagup(\{I_{i}\}\diagup\{I_{i}\}_{(B)})_{(C)} \\
=\{I_{i}\}_{I_{i}\cap(B\cup C)=\emptyset}\cup
\{(\mathcal{R}(\{I_{i}\}_{I_{i}\cap B=\emptyset,\,
I_{i}\cap C\not=\emptyset})\cup
(\mathcal{R}_{B}\setminus B))\setminus C\}. \\
\end{array}
$$
Noting $\mathcal{R}(\{I_{i}\}_{I_{i}\cap B=\emptyset,\,
I_{i}\cap C\not=\emptyset}\cap B=\emptyset$,
thus

$$
\begin{array}{c}
\mathcal{R}(\{I_{i}\}_{I_{i}\cap B=\emptyset,\,
	I_{i}\cap C\not=\emptyset})\cup
(\mathcal{R}_{B}\setminus B) \\
=(\mathcal{R}(\{I_{i}\}_{I_{i}\cap B=\emptyset,\,
I_{i}\cap C\not=\emptyset})\cup
\mathcal{R}_{B})\setminus B 
=\mathcal{R}_{B\cup C}\setminus B.
\end{array}
$$
In summary, we get

$$
(\{I_{i}\}\diagup\{I_{i}\}_{(B)})\diagup
(\{I_{i}\}\diagup\{I_{i}\}_{(B)})_{(C)}=
\{I_{i}\}_{I_{i}\cap(B\cup C)=\emptyset}\cup
(\mathcal{R}_{B\cup C}\setminus(B\cup C)).
$$
Above formula is (2.5) exactly.

$\mathbf{Case\,\,of\,\,\mathcal{R}_{B}
\cap \mathcal{R}_{C}=\emptyset:}$

In this situation,
we have $(\mathcal{R}_{B}
\setminus B)\cap C=\emptyset$, and
$\mathcal{R}(\{I_{i}\}_{I_{i}\cap B=\emptyset,\,
I_{i}\cap C\not=\emptyset}=\mathcal{R}(\{I_{i}\}_
{I_{i}\cap C\not=\emptyset})$, thus

$$
\begin{array}{c}
(\{I_{i}\}\diagup\{I_{i}\}_{(B)})\diagup(\{I_{i}\}\diagup\{I_{i}\}_{(B)})_{(C)} \\
=\{I_{i}\}_{I_{i}\cap(B\cup C)=\emptyset}\cup
\{\mathcal{R}_{B}\setminus B\}\cup
\{(\mathcal{R}(\{I_{i}\}_{I_{i}\cap B=\emptyset,\,
I_{i}\cap C\not=\emptyset})\setminus C\} \\
=\{I_{i}\}_{I_{i}\cap(B\cup C)=\emptyset}\cup\{\mathcal{R}_{B}
\setminus B\}\cup\{\mathcal{R}_{C}\setminus C\}.
\end{array}
$$
Now we reach the formula (2.6).

\end{proof}

\begin{corollary}
	
\begin{equation}
(\{I_{i}\}\diagup\{I_{i}\}_{(B)})\diagup(\{I_{i}\}\diagup\{I_{i}\}_{(B)})_{(C)}
=(\{I_{i}\}\diagup\{I_{i}\}_{(C)})\diagup(\{I_{i}\}\diagup\{I_{i}\}_{(C)})_{(B)}.
\end{equation}

\begin{equation}
(\{I_{i}\}\diagup\{I_{i}\}_{(B)})\diagup\{I_{i}\}_{(C)}
=(\{I_{i}\}\diagup\{I_{i}\}_{(C)})\diagup\{I_{i}\}_{(B)}, 
\end{equation}

\end{corollary}

Now we consider more general situation.

\begin{definition}
	Let $\{I_{i}\},\{J_{j}\}\in\mathcal{P}_{dis}^{2}(A)$, 
	$\mathcal{R}(\{J_{j}\})\subset\mathcal{R}(\{I_{i}\})$,
	we define the following quotient inductively,
	$$
	\begin{array}{c}
	\{I_{i}\}\diagup (J_{1})\doteq\{I_{i}\}\diagup\{I_{i}\}_{(J_{1})},
	\{I_{i}\}\diagup (J_{1},J_{2})\doteq(\{I_{i}\}\diagup (J_{1}))
	\diagup(\{I_{i}\}\diagup (J_{1}))_{(J_{2})}, \\
	\cdots,\{I_{i}\}\diagup (J_{1},\cdots,J_{k+1})\doteq
	(\{I_{i}\}\diagup (J_{1},\cdots,J_{k}))\diagup
	(\{I_{i}\}\diagup (J_{1},\cdots,J_{k}))_{(J_{k+1})},\cdots.
	\end{array}
	$$
\end{definition}

\begin{remark}
	By the formula (2,7) in corollary 2.1,
	it is easy to check that the quotient
	in definition 2.2 does not depend on 
	the order of $J_{j}$. Precisely,
	let $n=|\{J_{j}\}|$, $\tau\in\mathbb{S}_{n}$
	be a permutation on $\underline{n}$, then,
	
	$$
	\{I_{i}\}\diagup(J_{j})=
	\{I_{i}\}\diagup(J_{\tau(j)}).
	$$
\end{remark}

\begin{definition}
	Let $\{I_{i}\},\{J_{j}\}\in\mathcal{P}_{dis}^{2}(A)$, if
	$\mathcal{R}(\{J_{j}\})\subset\mathcal{R}(\{I_{i}\})$,
	and
	\begin{equation}
	\mathcal{R}_{\{I_{i}\},J_{j}}\cap\mathcal{R}_{\{I_{i}\},J_{j^{\prime}}}
	=\emptyset,\,j\not=j^{\prime},
	\end{equation}
	we call $\{J_{j}\}$ admits to $\{I_{i}\}$, denoted by
	$\{J_{j}\}\sqsubset\{I_{i}\}$.
\end{definition}

\begin{remark}
	The condition (2.9) is equivalent to the following conditions
	$$
	\mathcal{R}_{\{I_{i}\},J_{j}}\cap 
	J_{j^{\prime}}=\emptyset\,
	(or,\,
	\mathcal{R}_{\{I_{i}\},J_{j^{\prime}}}\cap 
	J_{j}=\emptyset),\,j\not=j^{\prime}.
	$$
	$\mathcal{R}_{\{I_{i}\},J_{j}}$ will be shortly
	denoted by $\mathcal{R}_{J_{j}}$ later.
\end{remark}

Actually, the fact 
$\{J_{j}\}_{j=1}^{p}\sqsubset\{I_{i}\}$
means that $\{I_{i}\}$ adapts to a 
decomposition as follows,

$$
\{I_{i}\}=\{I_{i}\}_{I_{i}\cap J=\emptyset}
\cup(\bigcup_{j}\{I_{i}\}_{I_{i}\cap J_{j}
	\not=\emptyset}),
$$
where $J=\mathcal{R}(\{J_{j}\})$.
Therefore, we have

$$
\begin{array}{c}
\{I_{i}\}\diagup(J_{j}) \\
=\{I_{i}\}_{I_{i}\cap J=\emptyset}
\cup(\bigcup_{j}\{I_{i}\}_{I_{i}\cap J_{j}
	\not=\emptyset}\diagup\{I_{i}\}_{(J_{j})}) \\
=\{I_{i}\}_{I_{i}\cap J=\emptyset}\cup
(\bigcup\limits_{j}\{\mathcal{R}_{J_{j}}\setminus J_{j}\}).
\end{array}
$$

Generally, $\{J_{j}\}$ may do not admit to $\{I_{i}\}$,
even though $\mathcal{R}(\{J_{j}\})\subset\mathcal{R}(\{I_{i}\})$.
But if we consider the quotient, the situation
can always be indeuced to the simple case.

\begin{proposition}
	Let $\{J_{j}\},\{I_{i}\}\in\mathcal{P}_{dis}^{2}(A)$,
	$\mathcal{R}(\{J_{j}\})\subset\mathcal{R}(\{I_{i}\})$.
	Then, there is a partition 
	$\{L_{l}\}\in\mathcal{P}_{dis}^{2}(A)$ satisfying
	
	\begin{itemize}
		\item $\{L_{l}\}\sqsubset\{I_{i}\}$.
		\item $\{J_{j}\}\subset\{L_{l}\}$, and
		$\mathcal{R}(\{J_{j}\})=\mathcal{R}(\{L_{l}\})$.
		\item If $\{K_{k}\}\in\mathcal{P}_{dis}^{2}(A)$
		satisfies $\{J_{j}\}\subset\{K_{k}\}\sqsubset\{I_{i}\}$,
		then we have $\{L_{l}\}\subset\{K_{k}\}$.
	\end{itemize}
	
	If we ignore the order in $\{L_{l}\}$, 
	$\{L_{l}\}$ is unique.
\end{proposition}

\begin{proof}
	In $\{J_{j}\}$ we define an equivalent
	relation as follows. Let $n=|\{J_{j}\}|$.
	For any $j,j^{\prime}\in\underline{n}$,
	we say $\mathcal{R}_{j}\sim \mathcal{R}_{j^{\prime}}$, 
	if there is a subset $\{j_{0},j_{1},\cdots,j_{m}\}$
	of $\underline{n}$, such that $j_{0}=j$,
	$j_{m}=j^{\prime}$, and
	$\mathcal{R}_{J_{j_{k}}}\cap\mathcal{R}_{J_{k+1}}
	\not=\emptyset$ ($k=0,\cdots,m-1$).
	It is obvious that $\sim$ is an equivalent
	relation.
	
	Under the equivalent relation defined above
	$\{J_{j}\}$ can be divided into the set
	of equivalent class, i.e. we have
	
	$$
	\{J_{j}\}=\bigcup\{J_{j}\}_{j\in E_{l}},
	$$
	where $\{E_{l}\}\in\mathbf{Part}(\underline{n})$,
	each $\{J_{j}\}_{j\in E_{l}}$ is an equivalent
	class under $\sim$.
	We take $L_{l}$ to be $L_{l}=\bigcup_{j\in E_{l}}J_{j}$.
	It is easy to check that 
	$\mathcal{R}(\{L_{l}\})=\mathcal{R}(\{J_{j}\})$
	and $\mathcal{R}_{L_{l}}\cap\mathcal{R}_{L_{l^{\prime}}}
	=\emptyset$ for $l\not= l^{\prime}$.
	From proposition 2.1 we know that
	
	$$
	\{I_{i}\}\diagup(J_{j})_{j\in E_{l}}=
	\{I_{i}\}\diagup\{I_{i}\}_{(\bigcup_{j\in E_{l}}J_{j})}
	=\{I_{i}\}\diagup\{I_{i}\}_{(L_{l})}.
	$$
	By the previous discussion, we can reach
	the formula
	$$
	\{I_{i}\}\diagup(J_{j})=\{I_{i}\}\diagup(L_{l}).
	$$
	
	Let $\{K_{k}\}\in\mathcal{P}_{dis}^{2}(A)$
	satisfy $\{J_{j}\}\subset\{K_{k}\}\sqsubset\{I_{i}\}$,
	then we can prove that for each $l$, there is
	$k$ such that $\mathcal{R}_{\{I_{i}\},L_{l}}\subset\mathcal{R}_{\{I_{i}\},K_{k}}$.
	On the other hand, $\{J_{j}\}\subset\{K_{k}\}$,
	thus we have $\{L_{l}\}\subset\{K_{k}\}$.
	
	Let $\{L^{\prime}_{l^{\prime}}\}\in
	\mathcal{P}_{dis}^{2}(A)$ satisfy the
	conditions same as ones of $\{L_{l}\}$, then
	both of $\{L_{l}\}\subset\{L^{\prime}_{l^{\prime}}\}$ 
	and $\{L^{\prime}_{l^{\prime}}\}\subset\{L_{l}\}$
	are valid,which implies $\{L_{l}\}=\{L^{\prime}_{l^{\prime}}\}$.
	
\end{proof}

We denote $\{L_{l}\}$ by
$\{L_{l}\}=\{J_{j}\}_{ad,\{I_{i}\}}$. From proposition 2.2,
when we discuss the quotient $\{I_{i}\}\diagup(J_{j})$,
we can always assume $\{J_{j}\}\sqsubset\{I_{i}\}$.

\begin{corollary}
Let $\{I_{i}\},\{J_{j}\},\{K_{k}\}
\in\mathcal{P}_{dis}^{2}(A)$ satisfying
$J\cap K=\emptyset$, $J,K\subset I$,
where $I=\mathcal{R}(\{I_{i}\})$,
$J=\mathcal{R}(\{J_{j}\})$,
$K=\mathcal{R}(\{K_{k}\})$. Then we have

$$
\begin{array}{c}
(\{J_{j}\}\cup\{K_{k}\})_{ad,\{I_{i}\}} \\
=(\{J_{j}\}_{ad,\{I_{i}\}}\cup\{K_{k}\}
_{ad,\{I_{i}\}})_{ad,\{I_{i}\}} \\
=(\{J_{j}\}\cup\{K_{k}\}_
{ad,\{I_{i}\}\diagup\{J_{j}\}})_{ad,\{I_{i}\}} \\
=(\{K_{k}\}\cup\{J_{j}\}_{ad,\{I_{i}\}
\diagup\{K_{k}\}})_{ad,\{I_{i}\}}.
\end{array}
$$
\end{corollary}

\subsection{Insertion of the partitions}

Now we turn to the discussion of insertion.
Let $\{I_{i}\},\{J_{j}\}\in \mathcal{P}_{dis}^{2}(A)$, 
$\mathcal{R}(\{I_{i}\})\cap\mathcal{R}(\{J_{j}\})=\emptyset$. 
We hope to define the insertion of $\{J_{j}\}$ into
$\{I_{i}\}$ at $I_{a}$.

\begin{definition}
Let $\{I_{i}\}_{1\leq i\leq m},\{J_{j}\}_{1\leq j\leq n}
\in \mathcal{P}_{dis}^{2}(A)$, 
$\mathcal{R}(\{I_{i}\})\cap\mathcal{R}(\{J_{j}\})=\emptyset$,
and $\iota:I_{a}\rightarrow \{J_{j}\}_{1\leq j\leq n}$ be a map
($1\leq a\leq m$). The insertion of $\{J_{j}\}_{j=1}^{n}$ into
$\{I_{i}\}_{i=1}^{m}$ at $I_{a}$ by $\iota$ is a partition
$\{I_{i}\}\circ_{a}^{\iota}\{J_{j}\}\in \mathcal{P}_{dis}^{2}(A)$,
where
\begin{equation}
\begin{array}{c}
\{I_{i}\}\circ_{I_{a}}^{\iota}\{J_{j}\} \\
=\{I_{1},\cdots,I_{a-1},J_{1}\cup\iota^{-1}(J_{1}),\cdots,
J_{n}\cup\iota^{-1}(J_{n}),I_{a+1},\cdots,I_{m}\}.
\end{array}
\end{equation}
$I_{a}$ is called the position of insertion, and $\iota$
is called insertion map.
\end{definition}

\begin{remark}
$\\$
\begin{itemize}
\item We can explain the insertion in terms of 
map-uinon. 
If we ignore the order of the partitions, the insertion
$\{I_{i}\}\circ_{a}^{\iota}\{J_{j}\}$ can be expressed
as $\{I_{i}\}_{i\not=a}\cup(I_{a}\sqcup_{\iota}\{J_{j}\})$.
In fact, the map-union is a special situation of insertion,
$B\sqcup_{\iota}\{J_{j}\}=\{B\}\circ_{B}^{\iota}\{J_{j}\}$
($B\cap\mathcal{R}(\{J_{j}\})=\emptyset$).
For simplicity, we can denote 
$\{I_{i}\}\circ_{a}^{\iota}\{J_{j}\}$
in the following intuitive way:
$$
\begin{array}{c}
\{\cdots,\hat{I_{a}},\cdots\}  \\
\uparrow \\
I_{a}\sqcup_{\iota}\{J_{j}\}.
\end{array}
$$

\item Particularly, we can always identify $\{I_{i}\}$ with
$\{I_{i}\}\cup\{\emptyset\}$, then the insertion of $\{J_{j}\}$ into
$\{I_{i}\}$ at $\emptyset$ is defined as
$\{I_{i}\}\cup\{J_{j}\}$ denoted by
$\{I_{i}\}\circ_{\emptyset}\{J_{j}\}$.
We call $\{I_{i}\}\circ_{\emptyset}\{J_{j}\}$ the
travail insertion. Let $\{K_{k}\}\in\mathcal{P}_{dis}^{2}(A)$
satisfying $(I\cup J)\cap\mathcal{R}(\{K_{k}\})=\emptyset$,
then, it is obvious that 
$(\{I_{i}\}\cup\{J_{j}\})\circ_{J_{b}}^{\iota}\{K_{k}\}
=\{I_{i}\}\cup(\{J_{j}\}\circ_{J_{b}}^{\iota}\{K_{k}\})$, or
$(\{I_{i}\}\circ_{\emptyset}\{J_{j}\})\circ_{J_{b}}^{\iota}\{K_{k}\}
=\{I_{i}\}\circ_{\emptyset}(\{J_{j}\}\circ_{J_{b}}^{\iota}\{K_{k}\})$.
\end{itemize}
\end{remark}

In the case of non-travail insertion we have:

\begin{proposition}
Let $\{I_{i}\},\{J_{j}\},\{K_{k}\}\in \mathcal{P}_{dis}^{2}(A)$
satisfying $I\cap J=\emptyset$,
$I\cup J=K$ ($I=\mathcal{R}(\{I_{i}\}),\,
J=\mathcal{R}(\{J_{j}\}),\,K=\mathcal{R}(\{K_{k}\})$),
$\{J_{j}\}=\{K_{k}\}_{(J)}$, 
$\mathcal{R}_{\{K_{k}\},J}
\setminus J\not=\emptyset$
.
Then
$$
\{K_{k}\}\diagup\{J_{j}\}=\{I_{i}\},
$$
if and only if, 
there is an insertion map $\iota:I_{a}\to \{J_{j}\}$
for some $I_{a}$ ($1\leq a\leq \#\{I_{i}\}$)
such that
$$
\{K_{k}\}=\{I_{i}\}\circ_{I_{a}}^{\iota}\{J_{j}\},\,\,(or\,\,
\{K_{k}\}=(\{K_{k}\}\diagup\{J_{j}\})\circ_{I_{a}}^{\iota}\{J_{j}\}).
$$
\end{proposition}

\begin{proof}
Let $\{K_{k}\}=\{I_{i}\}\circ_{I_{a}}^{\iota}\{J_{j}\}$.
By proposition 2.2 and remark 2.2 we have

$$
\{K_{k}\}=\{I_{i}\}_{i\not=a}\cup(I_{a}\sqcup_{\iota}\{J_{j}\}).
$$
Thus $\{K_{k}\}_{K_{k}\cap 
J=\emptyset}=\{I_{i}\}_{i\not=a}$,
$\{K_{k}\}_{K_{k}\cap J\not=\emptyset}
=I_{a}\sqcup_{\iota}\{J_{j}\}$
and $\{K_{k}\}_{(J)}=\{J_{j}\}$. Then we have

$$
\{K_{k}\}_{K_{k}\cap J\not=\emptyset}\diagup\{J_{j}\}
=(I_{a}\sqcup_{\iota}\{J_{j}\})\diagup\{J_{j}\}
=\{I_{a}\}.
$$
Finally, we get

$$
\{K_{k}\}\diagup\{J_{j}\}=
\{K_{k}\}_{K_{k}\cap J=\emptyset}\cup
(\{K_{k}\}_{K_{k}\cap J\not=\emptyset}\diagup\{J_{j}\})
=\{I_{i}\}_{i\not=a}\cup\{I_{a}\}=\{I_{i}\},
$$
i.e.

$$
(\{I_{i}\}\circ_{I_{a}}^{\iota}\{J_{j}\})\diagup\{J_{j}\}=\{I_{i}\}.
$$

Conversely, we assume $\{K_{k}\}$
satisfies $\{K_{k}\}\diagup\{J_{j}\}=\{I_{i}\}$.
Let $K=\mathcal{R}(\{K_{k}\})$,
$\mathcal{R}_{J}=
\mathcal{R}(\{K_{k}\}_{K_{k}\cap J\not=\emptyset})$,
then we have $K=I\cup J$, 
$\{J_{j}\}=\{K_{k}\}_{(J)}=
\{K_{k}\cap J\}_{K_{k}\cap J\not=\emptyset}$, and
$$
\{K_{k}\}_{K_{}\cap J=\emptyset}
\cup\{\mathcal{R}_{J}\setminus J\}=\{I_{i}\}.
$$
In order to recover $\{K_{k}\}$ by insertion,
the position of the insertion should be
taken to be $I_{a}=\mathcal{R}_{J}\setminus J$
for some $a$ ($1\leq a\leq \#\{I_{i}\}$). Thus
$\{K_{k}\}_{K_{}\cap J=\emptyset}=\{I_{i}\}_{i\not=a}$.
Let $\{K_{k}\}_{K_{k}\cap J\not=\emptyset}
=\{K_{k_{j}}\}$, and $K_{k_{j}}=J_{j}\cup L_{j}$,
where $J_{j}\cap L_{j}=\emptyset$,
we get a decomposition of $I_{a}$ that is 
$I_{a}=\bigcup_{j}L_{j}$.
Above decomposition define a map
$\iota:I_{a}\to\{J_{j}\}$, such that
$\iota^{-1}(J_{j})=L_{j}$. It is obvious
that $\{K_{k_{j}}\}=I_{a}\sqcup_{\iota}\{J_{j}\}$.
Up to now, we have proved

$$
\{K_{k}\}=\{I_{i}\}\circ_{I_{a}}^{\iota}\{J_{j}\}.
$$

\end{proof}

\begin{remark}
	In the first situation of proposition 2.3, the position of insertion is 
	the ideal part of quotient, and the insertion map is taken in
	a canonical way, thus we denote this insertion by symbol
	$\mathbf{\circ_{ideal_{B}}}$.
	Let $\{I_{i}\}\in \mathcal{P}_{dis}^{2}(A)$,
	$B\subset \mathcal{R}(\{I_{i}\})$.
	
	\begin{itemize}
		\item When $\mathcal{R}(\{I_{i}\}_{I_{i}\cap B\not=\emptyset})
		\setminus B\not=\emptyset$, we have
		$$
		(\{I_{i}\}\diagup\{I_{i}\}_{(B)})\circ_{ideal_{B}}\{I_{i}\}_{(B)}
		=\{I_{i}\},
		$$
		where the position of insertion is at 
		$\mathcal{R}(\{I_{i}\}_{I_{i}\cap B\not=\emptyset}
		\setminus B$, and the insertion map is taken to be
		$\iota^{-1}(I_{i}\cap B)=I_{i}\setminus B$ if
		$I_{i}\setminus B\not=\emptyset$.
		
		\item 
		When $\mathcal{R}(\{I_{i}\}_{I_{i}\cap B\not=\emptyset})=B$,
		we have
		$$
		\{I_{i}\}=(\{I_{i}\}\diagup\{I_{i}\}_{(B)})
		\circ_{\emptyset}\{I_{i}\}_{(B)},
		$$
		where the insertion is travail one.
	\end{itemize} 
\end{remark}

There is a conclusion about insertion and map-union
as follows.

\begin{proposition}
Let $\{I_{i}\},\{J_{j}\}\in\mathcal{P}_{dis}^{2}(A)$,
$B\subset A$, $I\cap J=\emptyset$, 
$B\cap(I\cup J)=\emptyset$, 
$I=\mathcal{R}(\{I_{i}\})$, $J=\mathcal{R}(\{J_{j}\})$.
For a pair $(f,\iota)$
there is an unique pair $(f^{\prime},\iota^{\prime})$
such that
$$
(B\sqcup_{f}\{I_{i}\})\circ_{f^{-1}(I_{a})\cup I_{a}}^{\iota}\{J_{j}\}=
B\sqcup_{f^{\prime}}(\{I_{i}\}\circ_{I_{a}}^{\iota^{\prime}}\{J_{j}\}),
$$
and vice-versa,
where $1\leq a\leq |\{I_{i}\}|$, $f:B\to\{I_{i}\}$, 
$f^{\prime}:B\to\{I_{i}\}\circ_{a}^{\iota^{\prime}}\{J_{j}\}$,
$\iota:f^{-1}(I_{a})\cup I_{a}\to\{J_{j}\}$, 
$\iota^{\prime}:I_{a}\to\{J_{j}\}$.
On the other hand, we also have

$$
(B\sqcup_{f^{\prime}}(\{I_{i}\}
\circ_{I_{a}}^{\iota^{\prime}}\{J_{j}\}))\diagup\{J_{j}\}
=B\sqcup_{f^{\prime}\circ p}
((\{I_{i}\}
\circ_{I_{a}}^{\iota^{\prime}}\{J_{j}\})\diagup\{J_{j}\})
=B\sqcup_{f^{\prime}\circ p}\{I_{i}\},
$$
where $p:\{I_{i}\}
\circ_{I_{a}}^{\iota^{\prime}}\{J_{j}\}\to\{I_{i}\}$ is
a projection satisfying

$$
p(I_{i})=I_{i}\,(i\not=a)\,\,
p((\iota^{\prime})^{-1}(J_{j})\cup J_{j})=I_{a}\,
(\forall j),\,\,f^{\prime}\circ p=f.
$$ 
\end{proposition}

\begin{proof}
Let $(f,\iota)$ be a given pair, we want to construct
the pair $(f^{\prime},\iota^{\prime})$ based on $(f,\iota)$.
By the definition of insertion and map-union we have

$$
(B\sqcup_{f}\{I_{i}\})
\circ_{f^{-1}(I_{a})\cup I_{a}}^{\iota}\{J_{j}\}=
\{f^{-1}(I_{i})\cup I_{i}\}_{i\not=a}\cup
((f^{-1}(I_{a})\cup I_{a})\sqcup_{\iota}\{J_{j}\}),
$$
and

$$
B\sqcup_{f^{\prime}}(\{I_{i}\}
\circ_{I_{a}}^{\iota^{\prime}}\{J_{j}\})=
\{f^{\prime\,-1}(I_{i})\cup I_{i}\}_{i\not=a}\cup
(B_{1}\sqcup_{f^{\prime}|_{B_{1}}}
(I_{a}\sqcup_{\iota^{\prime}}\{J_{j}\})),
$$
where $B_{1}=f^{\prime\,-1}(I_{a}\cup J)$.
The previous formulas imply  $f^{-1}(I_{i})=f^{\prime\,-1}(I_{i})$
($i\not=a$), and  $f^{-1}(I_{a})=f^{\prime\,-1}(I_{a}\cup J)$.
Thus we have $B_{1}=f^{-1}(I_{a})$ and
$f|_{B\setminus B_{1}}=f^{\prime}|_{B\setminus B_{1}}$.
If we regard $\iota,\iota^{\prime},f^{\prime}|_{B_{1}}$
as maps from some subsets to $\{J_{j}\}$,
then the first formula in proposition 2.4 means
that the formula $\iota=f^{\prime}|_{B_{1}}\cup\iota^{\prime}$
should be valid, i.e. we have

$$
f^{\prime}|_{B_{1}}=\iota|_{B_{1}},\,
\iota^{\prime}=\iota|_{I_{a}}.
$$ 

Conversely, starting from
$(f^{\prime},\iota^{\prime})$ we can determine
$(f,\iota)$ in similar way.

The second formula in proposition 2.4 is obviously valid.

\end{proof}

\begin{remark}
Observing the proof of proposition 2.4, in the procedure
of from $(f^{\prime},\iota^{\prime})$ to $(f,\iota)$,
$f^{\prime}$ determines $f$ uniquely, where $f^{\prime}$ and $f$
depend on ``$a$'' only. If we fix $f^{\prime}$,
there is an one-one corresponding between
$\iota$ and $\iota^{\prime}$.
\end{remark}

There is another version of proposition 2.1
in terms of the insertion.

\begin{proposition}
	Let $\{I_{i}\},\{J_{j}\},\{K_{k}\}
	\in\mathcal{P}_{dis}^{2}(A)$,
	$I\cap J=\emptyset$, $K\cap(I\cup J)=\emptyset$,
	where $I=\mathcal{R}(\{I_{i}\})$,
	$J=\mathcal{R}(\{J_{j}\})$,
	$K=\mathcal{R}(\{K_{k}\})$.
	Let $\{L_{l}\}=(\{K_{k}\}\circ_{K_{a}}^{\iota}\{I_{i}\})
	\circ_{\ast}^{\tau}\{J_{j}\}$, then
	
	\begin{itemize}
		\item 
		
		$$
		\mathcal{R}_{\{L_{l}\},I}\cap
		\mathcal{R}_{\{L_{l}\},J}=\emptyset
		\,\Leftrightarrow\,
		\{L_{l}\}=(\{K_{k}\}\circ_{K_{a}}^{\iota}\{I_{i}\})
		\circ_{K_{b}}^{\tau}\{J_{j}\},\,a\not=b.
		$$

		\item 
		$$
		\mathcal{R}_{\{L_{l}\},I}\cap
		\mathcal{R}_{\{L_{l}\},J}\not=\emptyset
		\,\Leftrightarrow\,
		\{L_{l}\}=(\{K_{k}\}\circ_{K_{a}}^{\iota}\{I_{i}\})
		\circ_{\iota^{-1}(I_{c})\cup I_{c}}^{\tau}\{J_{j}\}.
		$$
		In this situation, we have
		
		$$
		\{L_{l}\}=\{K_{k}\}\circ_{K_{a}}^{\iota^{\prime}}
		(\{I_{i}\}\circ_{I_{c}}^{\kappa}\{J_{j}\}),
		$$
		and
		
		$$
		\{L_{l}\}_{(I\cup J)}=
		\{I_{i}\}\circ_{I_{c}}^{\kappa}\{J_{j}\},
		$$
		for some $\iota^{\prime}$ and $\kappa$.
	\end{itemize}
\end{proposition}

\begin{proof}
$\mathbf{Case\,\,of\,\,\mathcal{R}_{\{L_{l}\},I}\cap
\mathcal{R}_{\{L_{l}\},J}=\emptyset:}$

Let $\mathcal{R}_{\{L_{l}\},I}\cap
\mathcal{R}_{\{L_{l}\},J}=\emptyset$,
then

$$
\{L_{l}\}=\{L_{l}\}_{L_{l}\cap(I\cup J)=\emptyset}
\cup\{L_{l}\}_{L_{l}\cap I\not=\emptyset}\cup
\{L_{l}\}_{L_{l}\cap J\not=\emptyset}.
$$
Noting $\{L_{l}\}_{(J)}=\{J_{j}\}$, by
proposition 2.1 and proposition 2.3 we have

$$
\{L_{l}\}\diagup\{J_{j}\}=
\{L_{l}\}_{L_{l}\cap(I\cup J)=\emptyset}
\cup\{L_{l}\}_{L_{l}\cap I\not=\emptyset}
\cup\{\mathcal{R}_{\{L_{l}\},J}\setminus J\}
=\{K_{k}\}\circ_{K_{a}}^{\iota}\{I_{i}\}.
$$
Furthermore, we now that $\{L_{l}\}_{(I)}=\{I_{i}\}$.
Thus

$$
(\{L_{l}\}\diagup\{J_{j}\})\diagup\{I_{i}\}=
\{L_{l}\}_{L_{l}\cap(I\cup J)=\emptyset}
\cup\{\mathcal{R}_{\{L_{l}\},I}\setminus I\}
\cup\{\mathcal{R}_{\{L_{l}\},J}\setminus J\}
=\{K_{k}\}.
$$
Above formula means that

$$
\{\mathcal{R}_{\{L_{l}\},I}\setminus I\}=
\{L_{l}\}_{L_{l}\cap I\not=\emptyset}\diagup\{I_{i}\}
=K_{a}\sqcup_{\iota}\{I_{i}\},
$$
and

$$
\{\mathcal{R}_{\{L_{l}\},J}\setminus J\}=
\{L_{l}\}_{L_{l}\cap J\not=\emptyset}\diagup\{J_{j}\}
=K_{b}\sqcup_{\tau}\{J_{j}\},
$$
for some $a$ and $b$ ($a\not=b$).

Conversely, let

$$
\{L_{l}\}=(\{K_{k}\}\circ_{K_{a}}^{\iota}\{I_{i}\})
\circ_{K_{b}}^{\tau}\{J_{j}\},\,a\not=b.
$$
Then, we have

$$
\{L_{l}\}=\{K_{k}\}_{k\not=a,b}\cup
(K_{a}\sqcup_{\iota}\{I_{i}\})\cup
(K_{b}\sqcup_{\tau}\{J_{j}\}),
$$
and
$\{L_{l}\}_{L_{l}\cap I\not=\emptyset}
=K_{a}\sqcup_{\iota}\{I_{i}\}$,
$\{L_{l}\}_{(I)}=\{I_{i}\}$,
$\{L_{l}\}_{L_{l}\cap J\not=\emptyset}
=K_{b}\sqcup_{\tau}\{J_{j}\}$,
$\{L_{l}\}_{(J)}=\{J_{j}\}$.
Above facts imply $\mathcal{R}_{\{L_{l}\},I}\cap
\mathcal{R}_{\{L_{l}\},J}=\emptyset$.

$\mathbf{Case\,\,of\,\,\mathcal{R}_{\{L_{l}\},I}\cap
\mathcal{R}_{\{L_{l}\},J}\not=\emptyset:}$

The conclusion in this situation is the
corollary of the one in the case of
$\mathcal{R}_{\{L_{l}\},I}\cap
\mathcal{R}_{\{L_{l}\},J}=\emptyset$.
Let

$$
\{L_{l}\}=\{K_{k}\}_{k\not=a}\cup
((K_{a}\sqcup_{\iota}\{I_{i}\})
\circ_{\iota^{-1}(I_{c})\cup I_{c}}^{\tau}\{J_{j}\}).
$$
By proposition 2.4 we have

$$
(K_{a}\sqcup_{\iota}\{I_{i}\})
\circ_{\iota^{-1}(I_{c})\cup I_{c}}^{\tau}\{J_{j}\}=
K_{a}\sqcup_{\iota^{\prime}}
(\{I_{i}\}\circ_{I_{c}}^{\kappa}\{J_{j}\}),
$$
for some $\iota^{\prime}$ and $\kappa$.
It is natural that we have

$$
\{L_{l}\}_{(I\cup J)}=
\{I_{i}\}\circ_{I_{c}}^{\kappa}\{J_{j}\}.
$$

\end{proof}

Proposition 2.5 can be generalized to
more general situation. Let
$\{I_{i}\},\{J_{j_{1}}^{(1)}\},\cdots,\{J_{j_{n}}^{(n)}
\}\in\mathcal{P}_{dis}^{2}(A)$,
$J^{(l)}\cap J^{(l^{\prime})}=\emptyset$ ($l\not=l^{\prime}$),
$I\cap(\bigcup_{l}J^{(l)})=\emptyset$,
$I=\mathcal{R}(\{I_{i}\})$, 
$J^{(l)}=\mathcal{R}(\{J_{j_{l}}^{(l)}\})$
($1\leq l\leq n$),

$$
\begin{array}{c}
\{K_{k_{1}}^{(1)}\}=\{I_{i}\}\circ_{I_{a}}^{\iota_{1}}\{J_{j_{1}}^{(1)}\},\,
\{K_{k_{2}}^{(2)}\}=\{K_{k_{1}}^{(1)}\}\circ_{K_{a_{1}}}^{\iota_{2}}
\{J_{j_{2}}^{(2)}\},\cdots, \\
\{K_{k}\}=\{K_{k_{n-1}}^{(n-1)}\}
\circ_{K_{a_{n-1}}}^{\iota_{n}}\{J_{j_{n}}^{(n)}\}.
\end{array}
$$

\begin{corollary}
$\mathcal{R}_{\{K_{k}\},J^{(l)}}\cap
\mathcal{R}_{\{K_{k}\},J^{(l^{\prime})}}
=\emptyset$ ($l\not=l^{\prime}$)
if and only if there are $a_{1},\cdots,a_{n}$
such that

$$
\{K_{k}\}=\{I_{i}\}_{i\not=a_{1},\cdots,a_{n}}\cup
(\bigcup\limits_{l=1}^{n}I_{a_{l}}\sqcup_{\iota_{l}}
\{J_{j_{l}}^{(l)}\}).
$$
\end{corollary}

\begin{remark}
The conclusion of corollary 2.2 can be
described by a different way, which is

$$
\{J^{(l)}\}_{l=1}^{n}\sqsubset\{K_{k}\}\,
\Leftrightarrow
\{K_{k}\}=\{I_{i}\}_{i\not=a_{1},\cdots,a_{n}}\cup
(\bigcup\limits_{l=1}^{n}I_{a_{l}}\sqcup_{\iota_{l}}
\{J_{j_{l}}^{(l)}\}),
$$
for somw $a_{1},\cdots,a_{n}$.
\end{remark}

\subsection{Coproduct}

Now we consider the coproduct
as an application of quotient.
At first we introduce 
$$
\mathcal{P}_{dis,\,k}^{2}(A)\subset
\underbrace{\mathcal{P}_{dis}^{2}(A)
\times\cdots\times\mathcal{P}_{dis}^{2}(A)}_{k-times},
$$

$$
\begin{array}{c}
(\{I_{i_{1}}^{(1)}\},\cdots,\{I_{I_{k}}^{(k)}\})
\in\mathcal{P}_{dis,\,k}^{2}(A)
\Longleftrightarrow\,
\{I_{i_{\lambda}}^{(\lambda)}\}\in \mathcal{P}_{dis}^{2}(A),\,
\lambda=1,\cdots,k, \\
\mathcal{R}(\{I_{i_{\lambda}}^{(\lambda)}\})
\cap\mathcal{R}(\{I_{i_{\mu}}^{(\mu)}\})=\emptyset,\,
1\leq\lambda<\mu\leq k.
\end{array}
$$

Similar to definition 2.3 we have

\begin{definition}
Let $\{J_{j}\}\in\mathcal{P}_{dis}^{2}(A)$,
$(\{I_{i_{1}}^{(1)}\},\cdots,\{I_{i_{k}}^{(k)}\})
\in\mathcal{P}_{dis,\,k}^{2}(A)$, we say
$\{J_{j}\}\sqsubset(\{I_{i_{1}}^{(1)}\},\cdots,\{I_{I_{k}}^{(k)}\})$,
if $\mathcal{R}(\{J_{j}\})\subset\bigcup_{\lambda=1}^{k}\mathcal{R}
(\{I_{i_{\lambda}}^{(\lambda)}\})$,
and $\{J_{j}\}_{J_{j}\cap\mathcal{R}(\{I_{i_{\lambda}}^{(\lambda)}\})
\not=\emptyset}\sqsubset\{I_{i_{\lambda}}^{(\lambda)}\}$,
($\lambda=1,\cdots,k$).
In this situation, we define

$$
\begin{array}{c}
(\{I_{i_{1}}^{(1)}\},\cdots,\{I_{I_{k}}^{(k)}\})\diagup(J_{j})
=(\{I_{i_{1}}^{(1)}\}\diagup(J_{j}),\cdots,\{I_{I_{k}}^{(k)}\diagup(J_{j})\}),
\end{array}
$$
where 
$$
\{I_{i_{\lambda}}^{(\lambda)}\}\diagup(J_{j})
=\{I_{i_{\lambda}}^{(\lambda)}\}\diagup(J_{j})_{J_{j}
	\subset\mathcal{R}(I_{i_{\lambda}}^{(\lambda)}\})},\,
\lambda=1,\cdots,k.
$$
\end{definition}

Similarly, we can discuss the insertion for the case
of $(\{I_{i_{1}}^{(1)}\},\cdots,\{I_{i_{k}}^{(k)}\})$
in a obvious way.

Let $\mathbb{K}$ be a field of characteristic zero,
$V_{\mathbb{K},A}=\bigoplus_{1\leq k\leq |A|}
\mathbf{Span}_{\mathbb{K}}(\mathcal{P}_{dis,\,k}^{2}(A))$.
Based on the above discussions we
now define coproduct on 
$V_{\mathbb{K},A}$,
$\bigtriangleup:V_{\mathbb{K},A}
\to V_{\mathbb{K},A}\otimes V_{\mathbb{K},A}$.

\begin{definition}
Let $\{I_{i}\}\in\mathcal{P}_{dis}^{2}(A)$,
$(\{I_{i_{1}}^{(1)}\},\cdots,\{I_{i_{k}}^{(k)}\})
\in\mathcal{P}_{dis,\,k}^{2}(A)$, we define
\begin{itemize}
\item

\begin{equation}
\begin{array}{c}
\bigtriangleup\{I_{i}\}=\emptyset\otimes\{I_{i}\}
+\{I_{i}\}\otimes\emptyset+ \\
\sum\limits_{\{J_{1},\cdots,J_{l}\}\sqsubset\{I_{i}\}}
(\{I_{i}\}_{(J_{1})}\cdots,\{I_{i}\}_{(J_{l})})\otimes\{I_{i}\}\diagup(J_{j}).
\end{array}
\end{equation}

\item
\begin{equation}
\begin{array}{c}
\bigtriangleup(\{I_{i_{1}}^{(1)}\},\cdots,\{I_{i_{k}}^{(k)}\}) \\
=\emptyset\otimes(\{I_{i_{1}}^{(1)}\},\cdots,\{I_{i_{k}}^{(k)}\})+
(\{I_{i_{1}}^{(1)}\},\cdots,\{I_{i_{k}}^{(k)}\})\otimes\emptyset+ \\
\sum\limits_{\{J_{j}\}\sqsubset(\{I_{i_{1}}^{(1)}\},\cdots,\{I_{i_{k}}^{(k)}\})}
(\{I_{i}\}_{(J_{1})}\cdots,\{I_{i}\}_{(J_{l})})\otimes
(\{I_{i_{1}}^{(1)}\},\cdots,\{I_{i_{k}}^{(k)}\})\diagup(J_{j}),
\end{array}
\end{equation}
where $\{I_{i}\}=\bigcup_{\lambda}\{I_{i_{\lambda}}^{(\lambda)}\}$.
\end{itemize}
\end{definition}

For associativity of coproduct defined in definition 2.6
we need the following conclusion.

\begin{theorem}
Let $\{I_{i}\},\{J_{j}\},\{K_{k}\}\in\mathcal{P}_{dis}^{2}(A)$,
$\{J_{j}\}\sqsubset\{I_{i}\}$, 
$\{K_{k}\}\sqsubset\{I_{i}\}\diagup(J_{j})$.
If we take 
$$
\{M_{\mu}\}=(\{J_{j}\}\cup\{K_{k}\})_{ad,\{I_{i}\}},
$$
then

$$
\{M_{\mu}\}=\{J_{j}\}_{\mathcal{R}_{J_{j}}\cap K=\emptyset}
\cup\{N_{k}\},
$$
where $N_{k}=K_{k}\cup(\bigcup_{\mathcal{R}_
{\{I_{i}\},J_{j}}\cap K_{k}\not=\emptyset}J_{j})$,
$K=\mathcal{R}(\{K_{k}\})$, $k=1,\cdots,|\{K_{k}\}|$.
$\{M_{\mu}\}$ satisfies the following conditions:

\begin{itemize}
\item $\{K_{k}\}=\{M_{\mu}\}\diagup(J_{j})$.	
	
\item $\{M_{\mu}\}\sqsubset\{I_{i}\}$.

\item $\{J_{j}\}\sqsubset(\cdots,\{I_{i}\}_{(M_{\mu})},\cdots)$.

\item $(\cdots,(\{I_{i}\}\diagup(J_{j}))_{(K_{k})},\cdots)
=(\cdots,\{I_{i}\}_{(M_{\mu})},\cdots)\diagup(J_{j})$.

\item $(\{I_{i}\}\diagup(J_{j}))\diagup(K_{k})
=\{I_{i}\}\diagup(M_{\mu})$.
\end{itemize}
\end{theorem}

\begin{proof}
At first, we prove

$$
\{M_{\mu}\}=\{J_{j}\}_{\mathcal{R}_{J_{j}}\cap K=\emptyset}
\cup\{N_{k}\}.
$$	
Let $\{L_{\lambda}\}=\{I_{i}\}\diagup(J_{j})$, then,
from definition of quotient we have

$$
\{L_{\lambda}\}=\{I_{i}\}_{I_{i}\cap J=\emptyset}
\cup\{\mathcal{R}_{J_{j}}\setminus J_{j}\}_{j=1}^{p},
$$
where $p=|\{J_{j}\}|$, $J=\mathcal{R}(\{J_{j}\}_{j=1}^{p})$,
and $\mathcal{R}_{J_{j}}=\mathcal{R}_{\{I_{i}\},J_{j}}$.
Recalling definition 2.3,
$\{K_{k}\}\sqsubset\{L_{\lambda}\}$ means that

$$
\mathcal{R}_{\{L_{l}\},K_{k}}\cap
\mathcal{R}_{\{L_{l}\},K_{k^{\prime}}}
=\emptyset,\,k\not=k^{\prime},
$$
On the other hand, according to 
$\{K_{k}\}$ there is the decomposition
of $\{L_{\lambda}\}$ as follows,

$$
\{L_{\lambda}\}=\{L_{\lambda}\}_{L_{\lambda}\cap K=\emptyset}
\cup\{L_{\lambda}\}_{L_{\lambda}\cap K_{1}\not=\emptyset}
\cup\cdots\cup\{L_{\lambda}\}_{L_{\lambda}\cap K_{q}\not=\emptyset},
$$	
where $K=\mathcal{R}(\{K_{k}\})$,
and $q=|\{K_{k}\}|$.
It is obvious that

$$
\{L_{\lambda}\}_{L_{\lambda}\cap K=\emptyset}=
\{I_{i}\}_{I_{i}\cap(J\cup K)=\emptyset}
\cup\{\mathcal{R}_{J_{j}}\setminus J_{j}\}
_{\mathcal{R}_{J_{j}}\cap K=\emptyset},
$$
and

$$
\{L_{\lambda}\}_{L_{\lambda}\cap K_{k}\not=\emptyset}=
\{I_{i}\}_{I_{i}\cap J=\emptyset,\,I_{i}\cap K_{k}\not=\emptyset}
\cup\{\mathcal{R}_{J_{j}}\setminus J_{j}\}
_{\mathcal{R}_{J_{j}}\cap K_{k}\not=\emptyset},
$$
where $k=1,\cdots,q$.
Notong $J\cap K=\emptyset$, thus

$$
(\mathcal{R}_{J_{j}}\setminus 
J_{j})\cap K_{k}\not=\emptyset
\,\Leftrightarrow\,
\mathcal{R}_{J_{j}}
\cap K_{k}\not=\emptyset.
$$
In summary, we know that for
each $\mathcal{R}_{J_{j}}$ there two
possibilities:

\begin{itemize}
\item $\mathcal{R}_{J_{j}}
\cap K=\emptyset$,
\item or, there is an unique $k$
such that $\mathcal{R}_{J_{j}}
\cap K_{k}\not=\emptyset$.
\end{itemize}
Up to now we have proved 

$$
\{M_{\mu}\}=\{J_{j}\}_{\mathcal{R}_{J_{j}}\cap K=\emptyset}
\cup\{N_{k}\}.
$$

We now begin to prove that $\{M_{\mu}\}$ satisfies
the conditions in theorem 2.1 step by step.

$\mathbf{\{K_{k}\}=\{M_{\mu}\}\diagup(J_{j}):}$

Let $\{G_{l}\}=\{J_{j}\}_{ad,\{M_{\mu}\}}$,
then we have

$$
\{G_{l}\}=\{J_{j}\}_{\mathcal{R}_{J_{j}}\cap K=\emptyset}
\cup(\bigcup\limits_{k}\{\bigcup\limits_
{\mathcal{R}_{J_{j}}\cap K_{k}\not=\emptyset}J_{j}\}),
$$
and $\{M_{\mu}\}\diagup(J_{j})=
\{M_{\mu}\}\diagup(G_{l})$.
It is easy to check that
$\{K_{k}\}=\{M_{\mu}\}\diagup(G_{l})$.

$\mathbf{\{M_{\mu}\}\sqsubset\{I_{i}\}}:$

Noting

$$
\{I_{i}\}_{I_{i}\cap N_{k}\not=\emptyset}=
\{I_{i}\}_{I_{i}\cap J=\emptyset,I_{i}\cap K_{k}\not=\emptyset}
\cup(\bigcup_{j}\{I_{i}\}_{I_{i}\cap J_{j}\not=\emptyset,
\mathcal{R}_{J_{j}}\cap K_{k}\not=\emptyset}),
$$
thus

$$
\mathcal{R}_{\{I_{i}\},N_{k}}=
\mathcal{R}_{\{L_{\lambda}\},K_{k}}\cup
(\bigcup_{\mathcal{R}_{J_{j}}\cap K_{k}\not=\emptyset}J_{j}),
\,\,k=1,\cdots,q.
$$
Because $\{K_{k}\}\sqsubset\{L_{\lambda}\}$, which means

$$
\mathcal{R}_{\{L_{\lambda}\},K_{k}}\cap
\mathcal{R}_{\{L_{\lambda}\},K_{k^{\prime}}}=\emptyset,
\,k\not=k^{\prime},
$$
hence, $\{N_{k}\}\sqsubset\{I_{i}\}$.

$\mathbf{\{J_{j}\}\sqsubset(\cdots,\{I_{i}\}_{(M_{\mu})},\cdots)}:$

It is enough for us to prove $\{J_{j}\}_{J_{j}\subset 
N_{k}}\sqsubset\{I_{i}\}_{(N_{k})}$,
$k=0,1\cdots,q$.
By the definition of $\{N_{k}\}$ we know
that for each $J_{j}$, $J_{j}\cap K\not=\emptyset$, 
there is a $N_{k}$ such that $J_{j}\subset N_{k}$. 
Thus we have

$$
\{I_{i}\}_{(N_{k})}=
\{I_{i}\cap K\}_{I_{i}\cap 
J=\emptyset,\,I_{i}\cap K_{k}\not=\emptyset}
\cup(\bigcup_{\mathcal{R}_{J_{j}}\cap K_{k}\not=
\emptyset}\{I_{i}\}_{(\mathcal{R}_{J_{j}}\cap(J\cup K))}).
$$
Noting $\mathcal{R}(\{I_{i}\}_{(N_{k})})=N_{k}$,
($1\leq k\leq q$), thus

$$
J_{j}\subset N_{k}\Leftrightarrow\mathcal{R}_{J_{j}}\cap K_{k}\not=
\emptyset.
$$
It is obvious that we have
$\{J_{j}\}_{J_{j}\subset N_{k}}\sqsubset\{I_{i}\}_{(N_{k})}$.

$\mathbf{(\cdots,(\{I_{i}\}\diagup(J_{j}))_{(K_{k})},\cdots)
=(\cdots,\{I_{i}\}_{(M_{\mu})},\cdots)\diagup(J_{j}):}$

Noting

$$
\{I_{i}\}_{(M_{\mu})}=\left\{
\begin{array}{cc}
\{I_{i}\}_{(J_{j})}, & \mathcal{R}_{J_{j}}\cap K=\emptyset, \\
\{I_{i}\}_{(N_{k})}, & others,
\end{array}
\right.
$$ 
thus, in the situation of 
$\mathcal{R}_{J_{j}}\cap K=\emptyset$,
$\{I_{i}\}_{(J_{j})}\diagup\{I_{i}\}_{(J_{j})}=\emptyset$,
and in the part of $\{I_{i}\}_{(N_{k})}$,
by the expressions of $\{I_{i}\}_{(N_{k})}$
and $\{L_{\lambda}\}_{L_{\lambda}\cap K_{k}\not=\emptyset}$,
we have

$$
\{L_{\lambda}\}_{(K_{k})}=
\{I_{i}\}_{(N_{k})}\diagup
(J_{j})_{J_{j}\subset N_{k}}.
$$

$\mathbf{(\{I_{i}\}\diagup(J_{j}))\diagup(K_{k})
=\{I_{i}\}\diagup(M_{\mu}):}$

We note that

$$
\{I_{i}\}\diagup(M_{\mu})=\{I_{i}\}_{I_{i}\cap M=\emptyset}
\cup(\bigcup_{j,\mathcal{R}_{J_{j}}\cap K=\emptyset}
\{I_{i}\}_{I_{i}\cap J_{j}\not=\emptyset}
\diagup\{I_{i}\}_{(J_{j})})
\cup(\bigcup_{k=0}^{q}\{I_{i}\}_{I_{i}\cap M_{k}\not=\emptyset}
\diagup\{I_{i}\}_{(M_{k})}),
$$
and

$$
\{L_{\lambda}\}\diagup(K_{k})=
\{L_{\lambda}\}_{L_{\lambda}\cap K=\emptyset}
\cup(\bigcup_{k=1}^{q}
\{L_{\lambda}\}_{L_{\lambda}\cap K_{k}\not=\emptyset}
\diagup\{L_{\lambda}\}_{(K_{k})}).
$$

It is obvious that

$$
\{L_{\lambda}\}_{L_{\lambda}\cap K=\emptyset}=
\{I_{i}\}_{I_{i}\cap M=\emptyset}\cup
(\bigcup_{j,\mathcal{R}_{J_{j}}\cap K=\emptyset}
\{I_{i}\}_{I_{i}\cap J_{j}\not=\emptyset}
\diagup\{I_{i}\}_{(J_{j})}).
$$

Noting $M_{k}=K_{k}\cup(\bigcup_
{\mathcal{R}_{J_{j}}\cap K_{k}\not=\emptyset}J_{j})$,
hence, we have

$$
\mathcal{R}_{\{L_{\lambda}\},K_{k}}\setminus K_{k}=
\mathcal{R}_{\{I_{i}\},M_{k}}\setminus M_{k}.
$$
Above formula implies

$$
\{L_{\lambda}\}_{L_{\lambda}\cap K_{k}\not=\emptyset}
\diagup\{L_{\lambda}\}_{(K_{k})}=
\{I_{i}\}_{I_{i}\cap M_{k}\not=\emptyset}
\diagup\{I_{i}\}_{(M_{k})},
$$
where $1\leq k\leq q$.

Up to now we have proved the proposition.

\end{proof}

As the corollary of theorem 2.1
we have:

\begin{proposition}
The coproduct in definition 2.6 satisfies

\begin{equation}
(\bigtriangleup\otimes id)\bigtriangleup
=(id\otimes\bigtriangleup)\bigtriangleup.
\end{equation}
\end{proposition}

\begin{proof}
	
For simplicity, we will replace the coproduct $\bigtriangleup$
in (2.13) by the reduce coproduct $\bigtriangleup^{\prime}$,
where 

$$
\bigtriangleup^{\prime}(\cdot)=\bigtriangleup(\cdot)
-(\cdot)\otimes\emptyset-\emptyset\otimes(\cdot).
$$	
Here we only give the proof in the situation of

$$
(\bigtriangleup^{\prime}\otimes id)\bigtriangleup^{\prime}(\{I_{i}\})
=(id\otimes\bigtriangleup^{\prime})\bigtriangleup^{\prime}(\{I_{i}\}),
\,\,\,(\ast)
$$
where $\{I_{i}\}\in\mathcal{P}^{2}_{dis}(A)$.
The general situation is similar.

Recalling the formula (2.11) in definition 2.6,
we have

$$
\bigtriangleup^{\prime}\{I_{i}\}=
\sum\limits_{\{J_{1},\cdots,J_{l}\}\sqsubset\{I_{i}\}}
(\{I_{i}\}_{(J_{1})}\cdots,\{I_{i}\}_{(J_{l})})
\otimes\{I_{i}\}\diagup(J_{j}).
$$
The right side of ($\ast$) shoula be
of the following form,

$$
\begin{array}{c}
(id\otimes\bigtriangleup^{\prime})
\bigtriangleup^{\prime}(\{I_{i}\}) \\
=\sum\limits_{\{J_{1},\cdots,J_{l}\}\sqsubset\{I_{i}\}}
(\{I_{i}\}_{(J_{1})}\cdots,\{I_{i}\}_{(J_{l})})
\otimes\bigtriangleup^{\prime}
(\{I_{i}\}\diagup(J_{j})) \\
=\sum\limits_{\{J_{j}\}\sqsubset\{I_{i}\},
\{K_{k}\}\sqsubset\{L_{\lambda}\}} 
(\{I_{i}\}_{(J_{j})})\otimes
(\{L_{\lambda}\}_{(K_{k})})\otimes(\{L_{\lambda}\}
\diagup(K_{k})),
\end{array}
$$
where $\{L_{\lambda}\}=\{I_{i}\}\diagup(J_{j})$.
Recalling proposition 2.8, we know
that there is $\{M_{\mu}\}\in\mathcal{P}^{2}_{dis}(A)$
satiafying the following conditions:

$$
\begin{array}{c}
\{M_{\mu}\}\sqsubset\{I_{i}\}$,
$\{J_{j}\}\sqsubset(\cdots,\{I_{i}\}_{(M_{\mu})},\cdots), \\
(\cdots,(\{I_{i}\}\diagup(J_{j}))_{(K_{k})},\cdots)
=(\cdots,\{I_{i}\}_{(M_{\mu})},\cdots)\diagup(J_{j}), \\
(\{I_{i}\}\diagup(J_{j}))\diagup(K_{k})
=\{I_{i}\}\diagup(M_{\mu}),
\end{array}
$$
then, we have

$$
\begin{array}{c}
(id\otimes\bigtriangleup^{\prime})
\bigtriangleup^{\prime}(\{I_{i}\}) \\
=\sum\limits_{\{M_{\mu}\}\sqsubset\{I_{i}\},
\{J_{j}\}\sqsubset(\{I_{i}\}_{(M_{\mu})})} 
(\{I_{i}\}_{(M_{\mu}\cap J_{j})})\otimes
(\{I_{i}\}_{(M_{\mu})})\diagup(J_{j})
\otimes(\{I_{i}\}\diagup(M_{\mu})).
\end{array}
$$
Because $\mathcal{R}(\{J_{j}\})
\subset\mathcal{R}(\{M_{\mu}\})$,
we have $(\{I_{i}\}_{(M_{\mu}\cap J_{j})})
=(\{I_{i}\}_{(J_{j})})$.
Finally, we have

$$
\begin{array}{c}
(id\otimes\bigtriangleup^{\prime})
\bigtriangleup^{\prime}(\{I_{i}\})\\
=\sum\limits_{\{M_{\mu}\}\sqsubset\{I_{i}\},
\{J_{j}\}\sqsubset(\{I_{i}\}_{(M_{\mu})})} 
(\{I_{i}\}_{(J_{j})})\otimes
(\{I_{i}\}_{(M_{\mu})})\diagup(J_{j})
\otimes(\{I_{i}\}\diagup(M_{\mu})) \\
=(\bigtriangleup^{\prime}\otimes id)
\bigtriangleup^{\prime}(\{I_{i}\}). 
\end{array}
$$

\end{proof}

\begin{remark}
Proposition 2.9 means that $V_{\mathbb{K},A}$ 
is a coalgebra under the coproduct in definition 2.6.
It is easy to check that for each $\{I_{i}\}\in\mathcal{P}_{dis}^{2}(A)$
(or $(\{I_{i_{1}}^{(1)}\},\cdots,\{I_{i_{k}}^{(k)}\})
\in\mathcal{P}_{dis,\,k}^{2}(A)$)
 $(\bigtriangleup^{\prime})^{m}(\{I_{i}\})=0$
 (or $(\bigtriangleup^{\prime})^{m}
 ((\{I_{i_{1}}^{(1)}\},\cdots,\{I_{i_{k}}^{(k)}\}))=0$)
for some positive integer $m$,
thus, from $V_{\mathbb{K},A}$ we can construct
a Hopf algebra in a standard way (see D. E. Radford \cite{8}).
\end{remark}

\section{Composition and Lie bracket}

Let $\mathbb{K}$ be a field of characteristic zero,
$\mathbb{K}(\mathcal{P}_{dis}^{2}(A))$ denote
the vector space over $\mathbb{K}$ spanned
by $\mathcal{P}_{dis}^{2}(A)$.
We now go to the discussion of the composition
between $\{I_{i}\}$ and $\{J_{j}\}$,
where $\{I_{i}\},\{J_{j}\}\in \mathcal{P}_{dis}^{2}(A)$,
$\mathcal{R}(\{I_{i}\})\cap\mathcal{R}(\{J_{j}\})=\emptyset$. 
When we focus on the composition, we restrict us
to consider the non-travail insertion only.
Similar to the cases of Feymman diagrams and
Kontsevich's graphs, we have:

\begin{definition}
We define
\begin{equation}
\{I_{i}\}\circ_{a}\{J_{j}\}=\sum\limits_{\iota}
\{I_{i}\}\circ_{I_{a}}^{\iota}\{J_{j}\},
\end{equation}
and
\begin{equation}
\{I_{i}\}\circ\{J_{j}\}=\sum\limits
_{a}\{I_{i}\}\circ_{a}\{J_{j}\}.
\end{equation}
\end{definition}

The composition $\circ$ defined by (3.2)
is not associative generally. Let 
$\{I_{i}\},\{J_{j}\},\{K_{k}\}\in \mathcal{P}_{dis}^{2}(A)$
satisfying $I\cap J=\emptyset$,
$K\cap(I\cup J)=\emptyset$, where
$I=\mathcal{R}(\{I_{i}\})$, $J=\mathcal{R}(\{J_{j}\})$
and $K=\mathcal{R}(\{K_{k}\})$.
We are interested in the difference
between $(\{I_{i}\}\circ\{J_{j}\})\circ\{K_{k}\}$
and $\{I_{i}\}\circ(\{J_{j}\}\circ\{K_{k}\})$.
In general, we have
$$
(\{I_{i}\}\circ\{J_{j}\})\circ\{K_{k}\}\not=
\{I_{i}\}\circ(\{J_{j}\}\circ\{K_{k}\}).
$$
Observing $(\{I_{i}\}\circ\{J_{j}\})\circ\{K_{k}\}$,
by (3.1) and (3.2) we know that

$$
(\{I_{i}\}\circ\{J_{j}\})\circ\{K_{k}\}=
(\sum\limits_{a,\iota}\{I_{i}\}
\circ_{I_{a}}^{\iota}\{J_{j}\})\circ\{K_{k}\},
$$
and

$$
(\{I_{i}\}\circ_{I_{a}}^{\iota}\{J_{j}\})\circ\{K_{k}\}
=\sum\limits_{b,\tau,a\not=b}
(\{I_{i}\}\circ_{I_{a}}^{\iota}\{J_{j}\})
\circ_{I_{b}}^{\tau}\{K_{k}\}+
\sum\limits_{c,\kappa}(\{I_{i}\}\circ_{I_{a}}^{\iota}\{J_{j}\})
\circ_{J_{c}\cup\iota^{-1}(J_{c})}^{\kappa}\{K_{k}\}.
$$
Recalling proposition 2.5 we know that

$$
\{I_{i}\}\circ_{I_{a}}^{\iota}\{J_{j}\})
\circ_{J_{c}\cup\iota^{-1}(J_{c})}^{\kappa}\{K_{k}\}
=\{I_{i}\}\circ_{I_{a}}^{\iota^{\prime}}
(\{J_{j}\}\circ_{J_{c}}^{\kappa^{\prime}}\{K_{k}\})
$$
for some $\iota^{\prime}$ and $\kappa^{\prime}$.
By proposition 2.4 we know that,
for fixed $I_{a}$ and $J_{c}$, there is an
one-one corresponding between $(\iota,\kappa)$
and $(\iota^{\prime},\kappa^{\prime})$.
Summarizing previous discussions, we reach
the following formula.

\begin{proposition}
\begin{equation}
(\{I_{i}\}\circ\{J_{j}\})\circ\{K_{k}\}=
\sum\limits_{a,b,a\not=b}\sum\limits_{\iota,\tau}
(\{I_{i}\}\circ_{I_{a}}^{\iota}\{J_{j}\})
\circ_{I_{b}}^{\tau}\{K_{k}\}+
\{I_{i}\}\circ(\{J_{j}\}\circ\{K_{k}\}).
\end{equation}
\end{proposition}

We introduce a compact notation

$$
\textlangle \{I_{i}\},\{J_{j}\},\{K_{k}\}\textrangle
=\sum\limits_{a,b,a\not=b}\sum\limits_{\iota,\tau}
(\{I_{i}\}\circ_{I_{a}}^{\iota}\{J_{j}\})
\circ_{I_{b}}^{\tau}\{K_{k}\},
$$
then the formula (3.3) can be rewritten as 

$$
(\{I_{i}\}\circ\{J_{j}\})\circ\{K_{k}\}=
\{I_{i}\}\circ(\{J_{j}\}\circ\{K_{k}\})+
\textlangle \{I_{i}\},\{J_{j}\},\{K_{k}\}\textrangle.
$$
Now we have

\begin{corollary}
$$
\textlangle \{I_{i}\},\{J_{j}\},\{K_{k}\}\textrangle
=\textlangle \{I_{i}\},\{K_{k}\},\{J_{j}\}\textrangle.
$$
\end{corollary}

\begin{proof}
Noting

$$
(\{I_{i}\}\circ_{I_{a}}^{\iota}\{J_{j}\})
\circ_{I_{b}}^{\tau}\{K_{k}\}=
\{I_{i}\}_{i\not=a,b}\cup(I_{a}\sqcup_{\iota}\{J_{j}\})
\cup(I_{b}\sqcup_{\tau}\{K_{k}\}),
$$
the conclusion will be implied immediately.
\end{proof}

\begin{corollary}
$$
\begin{array}{c}
\{I_{i}\}\circ(\{J_{j}\}\circ\{K_{k}\})-
(\{I_{i}\}\circ\{J_{j}\})\circ\{K_{k}\} \\
=\{I_{i}\}\circ(\{K_{k}\}\circ\{J_{j}\})-
(\{I_{i}\}\circ\{K_{k}\})\circ\{J_{j}\}.
\end{array}
$$
\end{corollary}

Based on the discussion as above we are able to define
the Lie bracket for partitions.

\begin{definition}
Let $\{I_{i}\},\{J_{j}\}\in \mathcal{P}_{dis}^{2}(A)$ 
with $\mathcal{R}(\{I_{i}\})\cap \mathcal{R}(\{J_{j}\})=\emptyset$,
their Lie bracket is defined to be
\begin{equation}
[\{I_{i}\},\{J_{j}\}]=\{I_{i}\}\circ\{J_{j}\}-\{J_{j}\}\circ\{I_{i}\}.
\end{equation}
\end{definition}

In order to prove the bracket (3.4) is Lie bracket,
it is enough for us to check that Jacobi identity is valid.

\begin{theorem}
The bracket (3.4) satisfies Jacobi identity.
\end{theorem}

\begin{proof}
Let 
$\{I_{i}\},\{J_{j}\},\{K_{k}\}\in \mathcal{P}_{dis}^{2}(A)$
satisfying $\mathcal{R}(\{I_{i}\})\cap\mathcal{R}(\{J_{j}\})\not=\emptyset$,
$\mathcal{R}(\{I_{i}\})\cap\mathcal{R}(\{K_{k}\})\not=\emptyset$,
$\mathcal{R}(\{J_{j}\})\cap\mathcal{R}(\{K_{k}\})\not=\emptyset$,
we want to prove Jacobi identity
$$
[\{I_{i}\},[\{J_{j}\},\{K_{k}\}]]+\,cocycle\,=0.
$$
By a straightforward calculation we get

$$
\begin{array}{c}
[\{I_{i}\},[\{J_{j}\},\{K_{k}\}]] \\
=\{I_{i}\}\circ(\{J_{j}\}\circ\{K_{k}\})
-\{I_{i}\}\circ(\{K_{k}\}\circ\{J_{j}\}) \\
-(\{J_{j}\}\circ\{K_{k}\})\circ\{I_{i}\}
+(\{K_{k}\}\circ\{J_{j}\})\circ\{I_{i}\},
\end{array}
$$

$$
\begin{array}{c}
[\{K_{k}\},[\{I_{i}\},\{J_{j}\}]] \\
=\{K_{k}\}\circ(\{I_{i}\}\circ\{J_{j}\})
-\{K_{k}\}\circ(\{J_{j}\}\circ\{I_{i}\}) \\
-(\{I_{i}\}\circ\{J_{j}\})\circ\{K_{k}\}
+(\{J_{j}\}\circ\{I_{i}\})\circ\{K_{k}\},
\end{array}
$$

$$
\begin{array}{c}
[\{J_{j}\},[\{K_{k}\},\{I_{i}\}]] \\
=\{J_{j}\}\circ(\{K_{k}\}\circ\{I_{i}\})
-\{J_{j}\}\circ(\{I_{i}\}\circ\{K_{k}\}) \\
-(\{K_{k}\}\circ\{I_{i}\})\circ\{J_{j}\}
+(\{I_{i}\}\circ\{K_{k}\})\circ\{J_{j}\}.
\end{array}
$$
With the help of corollary 3.2
we can get Jacobi identity.
\end{proof}

The previous discussions can be generalized
to more general situations. We assume that
every partition $\{I_{i}\}\in\mathcal{P}_{dis}^{2}(A)$ 
assigns to a function $f:\{I_{i}\}\rightarrow\mathbb{N}$. 
Now we consider the Lie bracket concerning
the pair $(\{I_{i}\},f)$. Let $(\{I_{i}\},f)$,
$(\{J_{j}\},g)$ be two pairs, where $\{I_{i}\},\{J_{j}\}
\in\mathcal{P}_{dis}^{2}(A)$, $\mathcal{R}(\{I_{i}\})
\cap\mathcal{R}(\{J_{j}\})=\emptyset$. Then the
insertion of two pairs is defined to be

\begin{equation}
(\{I_{i}\},f)\circ_{I_{a}}^{\iota}(\{J_{j}\},g)
=(\{I_{i}\}\circ_{I_{a}}^{\iota}\{J_{j}\},f\circ_{a}g),
\end{equation}
where $f\circ_{a}g$ is defined as follows:

$$
f\circ_{a}g(I_{i})=f(I_{i})\,(i\not=a),\,\,
f\circ_{a}g(\iota^{=1}(J_{j})\cup J_{j})
=g(J_{j}).
$$
Above formula shows that $f\circ_{a}g$ is
independent of $\iota$. Now we define the composition
of the pairs to be

\begin{equation}
(\{I_{i}\},f)\circ_{a}(\{J_{j}\},g)=
\sum\limits_{\iota}(\{I_{i}\},f)
\circ_{I_{a}}^{\iota}(\{J_{j}\},g),
\end{equation}
and

\begin{equation}
(\{I_{i}\},f)\circ_{pair}(\{J_{j}\},g)
=\sum\limits_{a}(-1)^{f(I_{a})}
(\{I_{i}\},f)\circ_{a}(\{J_{j}\},g).
\end{equation}

We hope that the Lie bracket arising from
the composition $\circ_{pair}$ will be well defined.
hence, we need to prove a conclusion similar to
corollary 3.2 is valid. Let $(\{I_{i}\},f)$,
$(\{J_{j}\},g)$, $(\{K_{k}\},h)$ be three pairs,
it is enough for us to consider the composition 
$(f\circ_{a}g)\circ_{b}h$. Similar to the situation
of proposition 3.1, it is necessary for us 
to discuss the following two possibilities:

\begin{itemize}
\item $\mathbf{Case\,\,of\,\,(\{I_{i}\}
\circ_{I_{a}}^{\iota}\{J_{j}\})
\circ_{I_{b}}^{\tau}\{K_{k}\}\,(a\not=b):}$
In this situation we have

$$
\left\{
\begin{array}{c}
((f\circ_{a}g)\circ_{b}h)(I_{i})
=(f\circ_{a}g)(I_{i})=
f(I_{i}),\,i\not=a,b, \\
((f\circ_{a}g)\circ_{b}h)(\iota^{-1}(J_{j})\cup J_{j})
=(f\circ_{a}g)(\iota^{-1}(J_{j})\cup J_{j})=g(J_{j}), \\
((f\circ_{a}g)\circ_{b}h)(\tau^{-1}(K_{k})\cup K_{k})=h(K_{k}).
\end{array}
\right.
$$
In this situation we have

$$
((f\circ_{a}g)\circ_{b}h)=
((f\circ_{b}h)\circ_{a}g).
$$

\item $\mathbf{Case\,\,of\,\,(\{I_{i}\}
\circ_{I_{a}}^{\iota}\{J_{j}\})
\circ_{\iota^{-1}(J_{c})\cup 
J_{c}}^{\kappa}\{K_{k}\}\,(a\not=b):}$
In this situation we have

$$
\left\{
\begin{array}{c}
((f\circ_{a}g)\circ_{b}h)(I_{i})=
(f\circ_{a}g)(I_{i})=f(I_{i}),\,i\not=a, \\
((f\circ_{a}g)\circ_{b}h)(\iota^{-1}(J_{j})\cup J_{j})
=(f\circ_{a}g)(\iota^{-1}(J_{j})\cup J_{j})=g(J_{j}),\,j\not=c, \\
((f\circ_{a}g)\circ_{b}h)(\kappa^{-1}(K_{k})\cup K_{k})
=h(K_{k}).
\end{array}
\right.
$$
In this situation, it is easy to check that

$$
(f\circ_{a}g)\circ_{b}h=
f\circ_{a}(g\circ_{c}h).
$$
\end{itemize}

Similar to previous discussions, we 
introduce the following notation

$$
\begin{array}{c}
\textlangle(\{I_{i}\},f),(\{J_{j}\},g),
(\{K_{k}\},h)\textrangle_{pair} \\
=\sum\limits_{a,b,a\not=b}(-1)^{f(I_{a})+f(I_{b})}
\sum\limits_{\iota,\tau}
((\{I_{i}\}\circ_{I_{a}}^{\iota}\{J_{j}\})
\circ_{I_{b}}^{\tau}\{K_{k}\},(f\circ_{a}g)\circ_{b}h).
\end{array}
$$
Now we reach the following conclusion:

\begin{proposition}

\begin{equation}
\begin{array}{c}
((\{I_{i}\},f)\circ_{pair}(\{J_{j}\},g))
\circ_{pair}(\{K_{k}\},h) \\
=\textlangle(\{I_{i}\},f),(\{J_{j}\},g),
(\{K_{k}\},h)\textrangle_{pair}+
(\{I_{i}\},f)\circ_{pair}((\{J_{j}\},g)
\circ_{pair}(\{K_{k}\},h)).
\end{array}
\end{equation}

\begin{equation}
\textlangle(\{I_{i}\},f),(\{J_{j}\},g),
(\{K_{k}\},h)\textrangle_{pair}=
\textlangle(\{I_{i}\},f),(\{K_{k}\},h),
(\{J_{j}\},g)\textrangle_{pair}.
\end{equation}
\end{proposition}

Up to now, we can easyly see
the conclusion similar to corollary 3.2
is valid.

\begin{corollary}
\begin{equation}
\begin{array}{c}
((\{I_{i}\},f)\circ_{pair}(\{J_{j}\},g))
\circ_{pair}(\{K_{k}\},h)
-(\{I_{i}\},f)\circ_{pair}((\{J_{j}\},g)
\circ_{pair}(\{K_{k}\},h)) \\
=((\{I_{i}\},f)\circ_{pair}(\{K_{k}\},h))
\circ_{pair}(\{J_{j}\},g)-
(\{I_{i}\},f)\circ_{pair}((\{K_{k}\},h)
\circ_{pair}(\{J_{j}\},g)).
\end{array}
\end{equation}
\end{corollary}

Summarizing the previous discussions,
we know that the following Lie bracket
will be well defined.

\begin{definition}
\begin{equation}
[(\{I_{i}\},f),(\{J_{j}\},g)]_{pair}=
(\{I_{i}\},f)\circ_{pair}(\{J_{j}\},g)-
(\{J_{j}\},g)\circ_{pair}(\{I_{i}\},f).
\end{equation}
\end{definition}

\section{Applications to graphs}

The construction discussed in section 2 and
section 3 is suitable for graphs, for example, Feymman
diagrams, Kontesvich's graphs and the ordinary graphs
in the sense of graphic theory.
For Feymman diagrams,
we follow the notations in Jean-Louis Loday and N. M. Nikolov \cite{5},
actually, which is one of our motivation for the
construction in this article.
For simplicity, we do not
consider the coloured Feymman diagrams.
As a preparation we introduce some
notations. Let $A$ be a finite set,
$\sigma:A\to A$ be an involution,
$\sigma^{2}=\sigma$. A subset $J\subset A$ is called 
a $\mathbf{\sigma-invariant}$ subset, if $J=\sigma(J)$.
Actually, for a subset $J\subset A$, it is easy to
check that $J\cap \sigma(J)$ and
$J\cup \sigma(J)$ are $\sigma-$invariant.
Roughly speaking, a graph can be viewed as
a finite set endowed with an involution
and a decomposition.

\subsection{The graphs in the sense of graphic theory}

\begin{definition}
\begin{itemize}
	
\item A ordinary graph is a pair $(\sigma,\{I_{i}\})$,
where $\{I_{i}\}\in\mathcal{P}_{dis}^{2}(A)$,
$\sigma$ is a map from $I$ to itself without 
fixed points, $\sigma^{2}=\sigma$,
where $I=\mathcal{R}(\{I_{i}\})$. 
A graph $(\sigma,\{I_{i}\})$ is also
denoted by $\varGamma_{\sigma,\{I_{i}\}}$,
or, $\varGamma_{\{I_{i}\}}$ for short.

\item If for any $I_{i}$ and $I_{i^{\prime}}$, there
are some positive integers $i_{0},i_{1},\cdots,i_{m}$
($i_{0}=i,\,i_{m}=i^{\prime}$), such that
$\sigma(I_{k})\cap I_{k+1}\not=\emptyset$
($0\leq k< m$), we say $\varGamma_{\{I_{i}\}}$
is a connected graph. Otherwise, we say
$\varGamma_{\{I_{i}\}}$ is disconnected.
We call $I_{i_{0}},I_{i_{1}},\cdots,I_{i_{m}}$
the chain in $\varGamma_{\{I_{i}\}}$ 
connecting $I_{i}$ and $I_{i^{\prime}}$.

\item Let $(\sigma,\{I_{i}\})$ be an ordinary graph,
$J$ be a $\sigma-$invariant subset of $I$.
A subgraph of $\varGamma_{\{I_{i}\}}$ 
related to $J$ is
a pair $(\sigma|_{J},\{I_{i}\}_{(J)})$,
denoted by $\varGamma_{J}\subset \varGamma_{\{I_{i}\}}$ also.
When $J=\mathcal{R}_{\{I_{i}\},J}\cap\sigma
(\mathcal{R}_{\{I_{i}\},J})$, we call $\varGamma_{J}$
a induced subgraph.

\end{itemize}
\end{definition}

\begin{remark}
\begin{itemize}
	
\item We call $\{I_{i}\}$ the set of vertices of
$\varGamma_{\{I_{i}\}}$, denoted by 
$\mathbf{Vert}(\varGamma_{\{I_{i}\}})$.
$I$ is called the total set of $\varGamma_{\{I_{i}\}}$.
For a subhraph $\varGamma_{J}$, we identify
$\mathbf{Vert}(\varGamma_{J})=\{I_{i}\cap J\}$
with $\{I_{i}\}_{I_{i}\cap J\not=\emptyset}$.

\item By the assumption of the involution $\sigma$,
we know that there is no $e\in I$ such that
$\sigma(e)=e$. Now we define an equivalent relation
in $I$ in the following way. We say $e_{1}\sim e_{2}$
($e_{1},e_{2}\in I$), if $e_{1}=\sigma(e_{2})$.
It is easy to check that $\sim$ is an equivalent relation,
and each equivalent class consists of two
pairs $(e,\sigma(e))$ and $(\sigma(e),e)$ ($e\in I$).
Each equivalent class assigns an edge of 
$\varGamma_{\{I_{i}\}}$, then $I\diagup\sim$ is the 
set of all edges of $\varGamma_{\{I_{i}\}}$.

\item Let $\tau:\underline{m}\rightarrow\underline{m}$
($m=|\{I_{i}\}|$) be a permutation, we identify
$(\sigma,\{I_{i}\})$ with $(\sigma,\{I_{\tau(i)}\})$.
\end{itemize}	
\end{remark}

\begin{lemma}
	Let $\varGamma_{\sigma,\{I_{i}\}}$ be a
	graph, $\varGamma_{J}$ and $\varGamma_{J^{\prime}}$
	are two connected subgraphs of $\varGamma_{\sigma,\{I_{i}\}}$.
	If $\mathcal{R}_{\{I_{i}\},J}\cap
	\mathcal{R}_{\{I_{i}\},J^{\prime}}\not=\emptyset$,
	then $\varGamma_{J\cup J^{\prime}}$ is a connected
	subgraph of $\varGamma_{\sigma,\{I_{i}\}}$.
\end{lemma}

\begin{proof}
	We need to prove a fact i.e. 
	for any $I_{i}$ and $I_{i^{\prime}}$ with
	$I_{i}\cap J\not=\emptyset$ and
	$I_{i^{\prime}}\cap J^{\prime}\not=\emptyset$
	(or $I_{i}\subset\mathcal{R}_{\{I_{i}\},J}$ and
	$I_{i^{\prime}}\subset\mathcal{R}_{\{I_{i}\},J^{\prime}}$)	
	there is a chain in $\varGamma_{J\cup J^{\prime}}$
	connecting $I_{i}\cap(J\cup J^{\prime})$ and 
	$I_{i^{\prime}}\cap(J\cup J^{\prime})$.
	Because of $\mathcal{R}_{\{I_{i}\},J}\cap
	\mathcal{R}_{\{I_{i}\},J^{\prime}}\not=\emptyset$,
	there is $I_{\alpha}$ such that
	$I_{\alpha}\subset\mathcal{R}_{\{I_{i}\},J}$ and 
	$I_{\alpha}\subset\mathcal{R}_{\{I_{i}\},J^{\prime}}$,
	which means $I_{\alpha}\cap J\not=\emptyset$ and
	$I_{\alpha}\cap J^{\prime}\not=\emptyset$.
	Noting both $\varGamma_{J}$ and $\varGamma_{J^{\prime}}$
	are connected, thus there is a chain
	in $\varGamma_{J}$ connecting $I_{i}\cap J$ with
	$I_{\alpha}\cap J$, and there is a chain in 
	$\varGamma_{J^{\prime}}$ connecting 
	$I_{i^{\prime}}\cap J^{\prime}$ with
	$I_{\alpha}\cap J^{\prime}$. Above two chains will
	result in a chain in $\varGamma_{J\cup J^{\prime}}$
	connecting $I_{i}\cap(J\cup J^{\prime})$ with
	$I_{i^{\prime}}\cap(J\cup J^{\prime})$. Therefore,
	we have proved the conclusion.
	
\end{proof}

\begin{proposition}
A graph $\varGamma_{\{I_{i}\}_{i=1}^{m}}$ 
is disconnected if and only if 
$\varGamma_{\{I_{i}\}_{i=1}^{m}}$ adapts
the following decomposition 

\begin{equation}
\varGamma_{\{I_{i}\}}=
\bigcup_{j}\varGamma_{J_{j}},
\end{equation}
where $\{J_{j}\}\in\mathcal{P}_{dis}^{2}(A)$
($2\leq|\{J_{j}\}|$) satisfying the following
conditions:

\begin{itemize}
\item $\mathcal{R}(\{J_{j}\})=
\mathcal{R}(\{I_{i}\})$, and
$\{J_{j}\}\sqsubset\{I_{i}\}$.
\item $\sigma(J_{j})=J_{j}$, 
$\varGamma_{J_{j}}$ is connected
subgraph for each $j$.
\item $\mathbf{Vert}(\varGamma_{J_{j}})\cap
\mathbf{Vert}(\varGamma_{J_{j^{\prime}}})=\emptyset,\,\,
j\not=j^{\prime}$.
\end{itemize}
Each $\varGamma_{J_{j}}$ is a connected component
of $\varGamma_{\{I_{i}\}}$.
If we ignore the order in the decomposition (4.1),
the decomposition (4,1) is unique.
\end{proposition}	

\begin{proof}
For $I_{i}$ and $I_{i^{\prime}}$, we 
say $I_{i}\sim I_{i^{\prime}}$, if
there is a chain connecting $I_{i}$ 
and $I_{i^{\prime}}$. $\sim$ is an equivalent
relation obviously. Thus there is a partition
$\{K_{k}\}\in\mathbf{Part}(\underline{m})$
such that 

$$
\{I_{i}\}_{i=1}^{m}=\bigcup\limits_{k}\{I_{i}\}_{i\in K_{k}},
$$
where each $\{I_{i}\}_{i\in K_{k}}$ is
an equivalent class. We take 
$J_{j}=\mathcal{R}(\{I_{i}\}_{i\in K_{j}})$.
We need to prove $\{J_{j}\}$ construted in
such a way satisfies the conditions in proposition. 
it is obvious that $\mathcal{R}(\{J_{j}\})=
\mathcal{R}(\{I_{i}\})$, and
$\{J_{j}\}\sqsubset\{I_{i}\}$. We now 
prove $\sigma(J_{j})=J_{j}$ for each $j$. It is enough for us
to prove for any $j$ and $j^{\prime}$
($j\not=j^{\prime}$), we have
$\sigma(J_{j})\cap J_{j^{\prime}}=\emptyset$.
Otherwise, there are $j$ and $j^{\prime}$
such that $\sigma(J_{j})\cap J_{j^{\prime}}
\not=\emptyset$. Then, there are $i\in K_{j}$
and $i^{\prime}\in K_{j^{\prime}}$,
$\sigma(I_{i})\cap I_{i^{\prime}}\not=\emptyset$,
which means $\{I_{i}\}_{i\in K_{j}}\cup 
\{I_{i}\}_{i\in K_{j^{\prime}}}$
should be included in some equivalent class.
That is a contradiction. The procedure to
construct $J_{j}$ impies each $\varGamma_{J_{j}}$
is a connected subgraph of $\varGamma_{\{I_{i}\}}$.
It is easy to check that for each connected
subgraph $\varGamma_{J}$ there is a $j$
such that $J\subset J_{j}$. Thus each 
$\varGamma_{J_{j}}$ is a connected component
of $\varGamma_{\{I_{i}\}}$. Furthermore,
we know that $\varGamma_{\{I_{i}\}}$ is disconnected
if and only if $|\{J_{j}\}|\geq 2$.
The uniquenessof the decomposition (4.1)
is obvious.

\end{proof}

\begin{remark}
	Let $\varGamma_{\sigma,\{I_{i}\}}$ be a
	graph, $\{J_{j}\}\in\mathcal{P}_{dis}^{2}(A)$.
	If $\{J_{j}\}\sqsubset\{I_{i}\}$ and
	$\sigma(J_{j})=J_{j}$ for any $j$,
	we say $\{J_{j}\}$ $\sigma-$admites to
	$\{I_{i}\}$ denoted by 
	$\{J_{j}\}\sqsubset_{\sigma}\{I_{i}\}$.	
\end{remark}

\begin{proposition}
	Let $\varGamma_{\sigma,\{I_{i}\}}$ be a connected
	graph, $\{J_{j}\}\in\mathcal{P}_{dis}^{2}(A)$ satisfy
	$J\subset I$ and $\sigma(J_{j})=J_{j}$ for any $j$, where
	$I=\mathcal{R}(\{I_{i}\})$, $J=\mathcal{R}(\{J_{j}\})$.
	Then, the subgraph $\varGamma_{J}$ adapts athe following decomposition

\begin{equation}
\varGamma_{J}=\bigcup_{l}\varGamma_{L_{l}},\,\,
\mathbf{Vert}(\varGamma_{L_{l}})\cap
\mathbf{Vert}(\varGamma_{L_{l^{\prime}}})
=\emptyset,\,\,l\not=l^{\prime},
\end{equation}
where $\{L_{l}\}=\{J_{j}\}_{ad,\{I_{i}\}_{(J)}}
\sqsubset\{I_{i}\}_{(J)}$, and each
$\varGamma_{L_{l}}$ is a connected
component of $\varGamma_{J}$.
\end{proposition}

\begin{proof}
According to proposition 4.1, for subgraph
$\varGamma_{J}$	there is the unique decomposition
based on its connected components.
Here we need to prove this decomposition
is exactly given by $\{L_{l}\}=\{J_{j}\}_{ad,\{I_{i}\}_{(J)}}$.
	
It is obvious that $\sigma(L_{l})=L_{(l)}$
for each $l$, thus each $\varGamma_{L_{l}}$
is a subgraph of $\varGamma_{J}$. By proposition 2.6
we know that $\{L_{l}\}\sqsubset\{I_{i}\}_{(J)}$.
In addition, noting $\mathcal{R}(\{L_{l}\})=J$,
it is easy to prove 
$\mathcal{R}_{\{I_{i}\}_{(J)},L_{l}}=L_{l}$
for each $l$. Recalling the procedure to construct $\{L_{l}\}$
in the proof of proposition 2.6, combining
with lemma 4.1, we know that $\varGamma_{L_{l}}$
is a connected subgraph of $\varGamma_{J}$ for each $l$.
It is obvious that as subgraphs of 
$\varGamma_{J}$, $\mathbf{Vert}(\varGamma_{L_{l}})
\cap\mathbf{Vert}(\varGamma_{L_{l^{\prime}}})=\emptyset$
($l\not=l^{\prime}$). Therefore, $\{L_{l}\}$
satisfies all conditions in proposition 4.2,
which means $\varGamma_{L_{l}}$ is a
connected component of $\varGamma_{J}$
for each $l$.

\end{proof}

We now turn to the quotient of the 
ordinary graphs. The discussions below will
follow the idea of Connes-Kriemer theory,
but a different description will be provided
based on the setting in this article.

\begin{definition}
Let $\varGamma_{\sigma,\{I_{i}\}}$ be a connected 
ordinary graph, $\varGamma_{J}$ be a connected 
subgraph of $\varGamma_{\sigma,\{I_{i}\}}$ determined by
a $\sigma-$invariant subset $J\subset I$.
We define the quotient of $\varGamma_{\sigma,\{I_{i}\}}$
by $\varGamma_{J}$ to be a pair
$(\sigma|_{I\setminus J},\{I_{i}\}\diagup\{J_{j}\})$,
denoted by $\varGamma_{\sigma,\{I_{i}\}}\diagup\varGamma_{J}$
also, where $I=\mathcal{R}(\{I_{i}\})$,
$\{J_{j}\}=\{I_{i}\}_{(J)}$.
\end{definition}

\begin{remark}
Let $\{J_{j}\}\in\mathcal{P}_{dis}^{2}(A)$ satisfy
$\mathcal{R}(\{J_{j}\})\subset\mathcal{R}(\{I_{i}\})$
and $\sigma(J_{j})=J_{j}$ for any $j$.
Then, it is natural for us to define
the quotient of $\varGamma_{\sigma,\{I_{i}\}}$
by the sequence of the subgraphs $\{\varGamma_{J_{j}}\}$
to be $(\sigma|_{I\setminus J},\{I_{i}\}\diagup(J_{j}))$,
where $J=\mathcal{R}(\{J_{j}\})$.
Above quotient can also be denoted by
$\varGamma_{\sigma,\{I_{i}\}}\diagup(\varGamma_{J_{j}})$.
Combining proposition 2.7 and proposition 4.1,
we have $\varGamma_{\sigma,\{I_{i}\}}\diagup(\varGamma_{J_{j}})
=\varGamma_{\sigma,\{I_{i}\}}\diagup(\varGamma_{L_{l}})$,
where $\{L_{l}\}=\{J_{j}\}_{ad,\{I_{i}\}_{(J)}}$
and each $\varGamma_{L_{l}}$ is an connected
component of $\varGamma_{J}$. Thus, when we 
discuss the quotient of the graphs, we can always
assume $\{J_{j}\}\sqsubset_{\sigma}\{I_{i}\}_{(J)}$.
\end{remark}

\begin{proposition}
Let $\varGamma_{\sigma,\{I_{i}\}}$ be a connected 
ordinary graph, $\varGamma_{J}$ be a connected 
subgraph of $\varGamma_{\sigma,\{I_{i}\}}$ determined by
a $\sigma-$invariant subset $J\subset I$. Then 
$\varGamma_{\sigma,\{I_{i}\}}\diagup\varGamma_{J}$
is a connected graph.
\end{proposition}

\begin{proof}
If $\mathcal{R}_{J}\setminus J=\emptyset$,
the conclusion is obviously valid.
Now we assume $\mathcal{R}_{J}\setminus J\not=\emptyset$.
Noting $\mathbf{Vert}(\varGamma_{\sigma,
\{I_{i}\}}\diagup\varGamma_{J})=
\{I_{i}\}_{I_{i}\cap J\not=\emptyset}
\cup\{\mathcal{R}_{J}\setminus J\}$,
because $\varGamma_{\sigma,\{I_{i}\}}$
is a connected graph, we know that
$\sigma(\mathcal{R}(\{I_{i}\}_{I_{i}\cap J\not=\emptyset}))
\cap\mathcal{R}_{J}\not=\emptyset$.
We need to prove $\sigma(\mathcal{R}
(\{I_{i}\}_{I_{i}\cap J\not=\emptyset}))
\cap(\mathcal{R}_{J}\setminus J)\not=\emptyset$.
Noting $\mathcal{R}(\{I_{i}\}_{I_{i}\cap J\not=\emptyset})
\cap J=\emptyset$, thus $\sigma(\mathcal{R}
(\{I_{i}\}_{I_{i}\cap J\not=\emptyset}))
\cap J=\emptyset$, which means
$\sigma(\mathcal{R}(\{I_{i}\}_{I_{i}\cap J\not=\emptyset}))
\cap\mathcal{R}_{J}\subset
\mathcal{R}_{J}\setminus J$.

\end{proof}

Combining proposition 4.2 and proposition 4.3
we have the following corollary.

\begin{corollary}
Let $\varGamma_{\sigma,\{I_{i}\}}$ be a connected
graph, $\{J_{j}\}\in\mathcal{P}_{dis}^{2}(A)$ satisfying
$\mathcal{R}(\{J_{j}\})\subset \mathcal{R}(\{I_{i}\})$ 
and $\sigma(J_{j})=J_{j}$ for any $j$.
Then $\varGamma_{\sigma,\{I_{i}\}}
\diagup(\varGamma_{J_{j}})$ is a connected graph.
\end{corollary}

In the situation of the ordinary graphs there is
an analogue of theorem 2.1.

\begin{proposition}
Let $\varGamma_{\sigma,\{I_{i}\}}$ be a
connected graph, $\{J_{j}\},\{K_{k}\}
\in\mathcal{P}_{dis}^{2}(A)$.
$\{I_{i}\},\,\{J_{j}\},\,\{K_{k}\}$
satisfy the following conditions:

\begin{itemize}
\item $\{J_{j}\}\sqsubset\{I_{i}\}$, 
$\{K_{k}\}\sqsubset\{I_{i}\}\diagup(J_{j})$.
\item $J_{j}$ is $\sigma-$invariant
and $\varGamma_{J_{j}}$ is a connected
subgraph in $\varGamma_{\sigma,\{I_{i}\}}$
for each $j$.

\item $K_{k}$ is $\sigma_{I\setminus J}-$
invariant and $\varGamma_{K_{k}}$
is a connected subgraph of
$\varGamma_{\sigma,\{I_{i}\}}
\diagup(\varGamma_{J_{j}})$ for each $k$,
where $I=\mathcal{R}(\{I_{i}\})$,
$J=\mathcal{R}(\{J_{j}\})$. 
\end{itemize}
Then there is a partition 
$\{M_{\mu}\}\in\mathcal{P}_{dis}^{2}(A)$
satisfying the following conditions:
\begin{itemize}
\item $\sigma(M_{\mu})=M_{\mu}$ and $\varGamma_{M_{\mu}}$ 
is connected for each $\mu$.

\item $\{M_{\mu}\}\sqsubset\{I_{i}\}$,
$\{J_{j}\}\sqsubset(\cdots,\{I_{i}\}_{(M_{\mu})},\cdots)$.
		
\item $(\cdots,(\{I_{i}\}\diagup(J_{j}))_{(K_{k})},\cdots)
=(\cdots,\{I_{i}\}_{(M_{\mu})},\cdots)\diagup(J_{j})$.
		
\item $(\varGamma_{\{I_{i}\}}\diagup(\varGamma_{J_{j}}))
\diagup(\varGamma_{K_{k}})=\varGamma_{\{I_{i}\}}
\diagup(\varGamma_{M_{\mu}})$.
\end{itemize}

\end{proposition}

\begin{proof}
The proof of proposition 4.4 is almost
same as the proof of proposition 2.7.
Here we only need to prove each $M_{\mu}$
is $\sigma-$invariant. Recalling proposition 
2.7, we know that

$$
\begin{array}{c}
\{M_{\mu}\}=\{J_{j}\}_{\mathcal{R}_{J_{j}}\cap K=\emptyset}
\cup\{N_{k}\}, \\
N_{k}=K_{k}\cup(\bigcup\limits_
{\mathcal{R}_{J_{j}}\cap K_{k}\not=\emptyset}J_{j}),\,\,
k=1,\cdots,|\{K_{k}\}|,
\end{array}
$$
where $K=\mathcal{R}(\{K_{k}\})$.
It is easy to check that
$\sigma(M_{\mu})=M_{\mu}$ for each $\mu$.

\end{proof}

We now consider the insertion of the ordinary
graphs.  

\begin{definition}
Let $(\sigma,\{I_{i}\}_{i=1}^{m})$, 
$(\lambda,\{J_{j}\}_{j=1}^{n})$
be two connected ordinary graphs,
with structure maps
$\sigma$ and $\lambda$ respectively,
$I=\mathcal{R}(\{I_{i}\}$,
$J=\mathcal{R}(\{J_{j}\}$,
$I\cap J=\emptyset$. Then we define
the insertion of
$\varGamma_{\{J_{j}\}}$ into $\varGamma_{\{I_{i}\}}$ 
at $I_{a}$ by $\iota$ to be a pair 
$(\delta,\{I_{i}\}\circ_{a}^{\iota}\{J_{j}\})$ 
denoted by $\varGamma_{\{I_{i}\}}
\circ_{a}^{\iota}\varGamma_{\{J_{j}\}}$, or,
$\varGamma_{\{I_{i}\}\circ_{a}^{\iota}
\{J_{j}\}}$ also, where
$1\leq a\leq m$, $\iota:I_{a}\to \{J_{j}\}$, 
$\delta|_{I}=\sigma$, $\delta|_{J}=\lambda$.
\end{definition}

It is obvious that we have:

\begin{proposition}
Let $\varGamma_{\sigma,\{I_{i}\}}$,
$\varGamma_{\delta,\{J_{j}\}}$ be
two connected graphs. Then,
$\varGamma_{\{I_{i}\}}
\circ_{a}^{\iota}\varGamma_{\{J_{j}\}}$
is a connected graph also.
\end{proposition}

Based on the discussions in section 3,
we can prove there is a well-defined 
Lie bracket structure on the
vector space $\mathbf{Span}_{\mathbb{K}}
(\mathcal{G}_{c})$, where $\mathcal{G}_{c}$
denotes the set of all connected graphs,
$\mathbb{K}$ is a field of characteristic zero.
Here we only consider the composition
$\circ$ in definition 3.1 for simiplicity.
The composition will be definied in a natural
way. Let $\varGamma_{\sigma,\{I_{i}\}}$,
$\varGamma_{\delta,\{J_{j}\}}$ be
two connected graphs. We have:

\begin{equation}
\varGamma_{\sigma,\{I_{i}\}}\circ
\varGamma_{\delta,\{J_{j}\}}=
\sum\limits_{a,\iota}
\varGamma_{\{I_{i}\}}
\circ_{a}^{\iota}\varGamma_{\{J_{j}\}}.
\end{equation}
The Lie bracket should be defined to be:

\begin{equation}
[\varGamma_{\sigma,\{I_{i}\}},
\varGamma_{\delta,\{J_{j}\}}]=
\varGamma_{\sigma,\{I_{i}\}}\circ
\varGamma_{\delta,\{J_{j}\}}-
\varGamma_{\delta,\{J_{j}\}}\circ
\varGamma_{\sigma,\{I_{i}\}}.
\end{equation}

\subsection{Feynman diagrams}

\begin{definition}
\begin{itemize}
\item A Feymman diagram is a pair $(\sigma,\{I_{i}\})$,
where $\{I_{i}\}\in\mathcal{P}_{dis}^{2}(A)$,
$\sigma$ is a map from $I$ to itself,
$\sigma^{2}=\sigma$, where $I=\mathcal{R}(\{I_{i}\})$.
A Feymman diagram $(\sigma,\{I_{i}\})$ is also 
denoted by $\Gamma_{\sigma,\{I_{i}\}}$,
or, $\Gamma_{\{I_{i}\}}$ for short.
$I=\mathcal{R}(\{I_{i}\})$ is called the 
total set of edges.

\item We call $I_{ext}=\{e\in I|\sigma(e)=e\}$ 
the set of external lines and 
$I_{int}=(I\setminus I_{ext})\diagup\sim$
the set of internal lines.
We say a Feynman diagram
$(\sigma,\{I_{i}\})$ is connected if
$(\sigma|_{I\setminus I_{ext}},\{I_{i}\}_{(I\setminus I_{ext})})$
is connected.
\end{itemize}	
\end{definition}

\begin{remark}
\begin{itemize}
\item 
It is obvious that $I\setminus I_{ext}$ is 
$\sigma-$invariant, and $\sigma|_{I\setminus I_{ext}}$
has no fixed points, thus $(\sigma|_{I\setminus I_{ext}},
\{I_{i}\}_{(I\setminus I_{ext})})$ is an ordinary
graph denoted by $\varGamma_{\{I_{i}\},int}$.
$\varGamma_{\{I_{i}\},int}$ is a Feynman diagram
without external lines. A general Feymman diagram can be
regarded as an extension of an ordinary graph.
In fact, let $(\sigma,\{I_{i}\})$ be an ordinary graph,
i.e. $I_{ext}=\emptyset$, now we want to add
the set of external lines $\tilde{I}_{ext}$ to $I$
by a map $f:\tilde{I}_{ext}\to \{I_{i}\}$.
The new Feymman diagram is 
naturally chosen to be the pair
$(\omega,\tilde{I}_{ext}\sqcup_{f}\{I_{i}\})$,
the new structure map $\omega$ is defined to be
$\omega|_{I}=\sigma$, $\omega|_{\tilde{I}_{ext}}=id$.
A Feymman diagram 
$(\sigma,\{I_{i}\})$ with $I_{ext}\not=\emptyset$ can be rewritten as
$(\sigma,I_{ext}\sqcup_{f}\{I_{i}\}_{(I\setminus I_{ext})})$,
which is an extension of 
$(\sigma|_{I\setminus I_{ext}},\{I_{i}\}_{(I\setminus I_{ext})})$,
where $f^{-1}(I_{i}\setminus I_{ext})=I_{i}\cap I_{ext}$.

\item We call $\{I_{i}\}$ the set of vertices of
$\Gamma_{\{I_{i}\}}$ denoted by 
$\mathbf{vert}(\Gamma_{\{I_{i}\}})$.
We identify the vertices of $\Gamma_{\{I_{i}\}}$
with ones of $\varGamma_{\{I_{i}\},int}$, i.e.
we have $\mathbf{vert}(\Gamma_{\{I_{i}\}})=
\mathbf{Vert}(\varGamma_{\{I_{i}\},int})$.
Actually, for a graph $\Gamma_{\{I_{i}\}}$,
there are two structure maps, $\sigma$ and projection
$p:\mathcal{R}(\{I_{i}\})\rightarrow \{I_{i}\}$,
$p(e)=I_{i}\Longleftrightarrow e\in I_{i}$.

\end{itemize}
\end{remark}

\begin{definition} 
Let $(\sigma,\{I_{i}\})$ be a connected Feymman diagram,
$J\subset I$ be a subset, where $I=\mathcal{R}(\{I_{i}\})$.
We call the pair 
$(\sigma_{J},\{I_{i}\}_{I_{i}\cap J\not=\emptyset})$ 
is a sub-diagram of $(\sigma,\{I_{i}\})$,
if $J$ satisfies the following conditions:

\begin{itemize}
	
\item 

$$
\sigma(J)\cap\mathcal{R}_{J}
=J\cap\sigma(\mathcal{R}_{J}),
$$
	
\item
If $I_{i}\cap J\not=\emptyset$, then
$I_{i}\cap[(J\cap\sigma(\mathcal{R}_{J}))
\setminus I_{ext}]\not=\emptyset$,
\end{itemize}
where $\mathcal{R}_{J}=
\mathcal{R}(\{I_{i}\}_{I_{i}\cap J\not=\emptyset})$
is the tatol set of 
$(\sigma_{J},\{I_{i}\}_{I_{i}\cap J\not=\emptyset})$.
The structure map
$\sigma_{J}:\mathcal{R}_{J}\to \mathcal{R}_{J}$ 
is defined to be:
$$
\sigma_{J}(e)=\left\{
\begin{array}{cc}
\sigma(e), &
e\in J\cap\sigma(\mathcal{R}_{J}), \\
e, & e\in \mathcal{R}_{J}\setminus
(J\cap\sigma(\mathcal{R}_{J})).
\end{array}
\right.
$$
The sub-diagram is also denoted by 
$\Gamma_{J}\subset \Gamma_{\{I_{i}\}}$.
When $J\cap\sigma(\mathcal{R}_{J})=\mathcal{R}_{J}
\cap\sigma(\mathcal{R}_{J})$ we call $\Gamma_{J}$
a subgraph.
\end{definition}

\begin{remark}
\begin{itemize}
\item The first condition in definition 4.5 
means that $J\cap\sigma(\mathcal{R}_{J})$ is 
$\sigma-$invariant.
\item The second condition means that
each vertex of $\Gamma_{J}$ is the
endpoint of at least one internal
line in $\Gamma_{J}$.
\item For subgraph $\Gamma_{J}$, $\mathcal{R}_{J}
\setminus(J\cap\sigma(\mathcal{R}_{J}))$ plays
the role of external lines of $\Gamma_{J}$.
Here, the subset $[(\mathcal{R}_{J}
\cap\sigma(\mathcal{R}_{J}))\setminus
(J\cap\sigma(\mathcal{R}_{J}))]\setminus I_{ext}$ 
is regarded as a subset of external lines 
of $\Gamma_{J}$. Actually, the external
momenta corresponding to this subset
will be canceled each other. 

\item The set of external lines in 
$\Gamma_{J}$ should be 

$$
(\mathcal{R}_{J}\setminus
(J\cap\sigma(\mathcal{R}_{J})))\cup
(\mathcal{R}_{J}\cap I_{ext}).
$$
The set of internal lines in $\Gamma_{J}$ is

$$
((J\cap\sigma(\mathcal{R}_{J}))\setminus I_{ext})
\diagup\sim.
$$
\end{itemize}
\end{remark}

Let $(\sigma,\{I_{i}\})$ be a Feymman diagram,
we call a $\sigma-$invariant subset
$J$ is internal, if $J\cap I_{ext}=\emptyset$.
We can prove that each sub-diagram
of $(\sigma,\{I_{i}\})$ can be uniquely 
determined by an internal $\sigma-$invariant subset.

\begin{proposition}
	There is an one-one corresponding
	between the sub-diagrams and internal
	$\sigma-$invariant subsets. 
\end{proposition}

\begin{proof}
Let $(\sigma,\{I_{i}\})$ be a Feymman diagram,
$J\subset I=\mathcal{R}(\{I_{i}\})$ satisfy
the conditions in definition 4.5. Then,
$J^{\prime}=(J\cap\sigma(\mathcal{R}_{J}))
\setminus I_{ext}$ is an internal 
$\sigma-$invariant subset. Due to the 
first condition in definition 4.5,
we know that $\mathcal{R}_{J}=
\mathcal{R}_{J^{\prime}}$. Thus
$\varGamma_{J}=\varGamma_{J^{\prime}}$.
Conversely, let $J\subset I$ be an
$\sigma-$invariant subset, then,
$J$ satisfies all conditions in
definition 4.5. Thus $J$ determines
a sub-diagram of $(\sigma,\{I_{i}\})$.

\end{proof}

\begin{remark}
\begin{itemize}
\item Let $(\sigma,\{I_{i}\})$ be a Feymman diagram,
and $J\subset I=\mathcal{R}(\{I_{i}\})$ 
be an internal $\sigma-$invariant subset.
Then, the set of external lines in $\varGamma_{J}$
just be $\mathcal{R}_{J}\setminus J$, and
the set of internal lines in $\varGamma_{J}$
is $J\diagup\sim$.
\item By the definition 4.1 and 4.5, we know that there is
much difference between the sub-diagrams
of Feynman diagrams and subgraphs of the 
ordinary graphs. Noting $(\sigma|_{I\setminus I_{ext}}, \\
\{I_{i}\}_{(I\setminus I_{ext})})$ can be 
regarded as a ordinary graph, the conclusion of
proposition 4.6 means that there is an one-one
correspondence between the sub-diagrams of
Feymman diagram $(\sigma,\{I_{i}\})$ and
the subgraphs of $(\sigma|_{I\setminus I_{ext}},
\{I_{i}\}_{(I\setminus I_{ext})})$ in the sense
of the ordinary graphs. Thus, when we discuss
the sub-diagrams of Feynman diagrams, the
discussion can be reduced to the situations
of the ordinary graphs.
\end{itemize}	 
\end{remark}

From proposition 4.2 and proposition 4.6
we have,

\begin{proposition}
Let $\varGamma_{\{I_{i}\}}$ be a Feymman diagram,
$\varGamma_{J}$ be a sub-diagram, where $J$
is an internal $\sigma-$invariant subset.  
If $\varGamma_{J}$ is a disconnected sub-diagram, 
then $\varGamma_{J}$ admits to the decomposition
as follows:

$$
\varGamma_{J}=\bigcup_{j}\varGamma_{J_{j}},
$$
where $\{J_{j}\}$ satisfies the following
conditions:

\begin{itemize}
\item $\{J_{j}\}\in\mathbf{Part}(J)$,
$|\{J_{j}\}|\geq 2$,
\item each $J_{j}$ is an internal
$\sigma-$invariant subset,
\item $\{J_{j}\}\sqsubset\{I_{i}\}$,
\item each $\varGamma_{J_{j}}$ is a connected
component of $\varGamma_{J}$. 
\end{itemize}

The decomposition of $\varGamma_{J}$ is unique.
\end{proposition}

We now consider the quotient of Feynman diagrams.

\begin{definition}
Let $\varGamma_{\{I_{i}\}}$ be a connected 
Feymman diagram, $\varGamma_{J}$ be a connected proper
sub-diagram of $\Gamma_{\{I_{i}\}}$ determined by
a internal $\sigma-$invariant subset $J\subset I$.
We define the quotient of $\varGamma_{\{I_{i}\}}$
by $\varGamma_{J}$ to be a pair
$(\sigma|_{I\setminus J},\{I_{i}\}\diagup\{J_{j}\})$,
denoted by $\varGamma_{\{I_{i}\}}\diagup\varGamma_{J}$
also, where $\{J_{j}\}=\{I_{i}\}_{(J)}$. 	
\end{definition}
	
\begin{remark}
$\\$
\begin{itemize}
		
\item Comparing with definition 4.2, in the
situation of $I_{ext}=\emptyset$, the quotient of Feynman
diagrams is same as the quotient of the ordinary
graphs exactly. Recalling the contents about $f-$union 
in section 2 and discussions in remark 4.1, the quotient 
of general Feymman diagrams can be reduced to the situation of the
ordinary graphs. Actually, a Feymman diagram
$(\sigma,\{I_{i}\})$ can be expressed as
$(\sigma, I_{ext}\sqcup_{f}\{I_{i}\}_{(I\setminus I_{ext})})$,
where $f:I_{ext}\to\{I_{i}\}_{(I\setminus I_{ext})}$ satisfies
$f^{-1}(I_{i}\cap(I\setminus I_{ext}))=I_{i}\cap I_{ext}$.
Therefore, for an internal $\sigma-$invariant subset $J$, 
from the viewpoint of partition, the quotient
$\Gamma_{\{I_{i}\}_{(I\setminus I_{ext})}\diagup\{I_{i}\}_{(J)}}$
results in a projection $p_{J}:\{I_{i}\}_{(I\setminus I_{ext})}
\to\{I_{i}\cap(I\setminus I_{ext})\}_{I_{i}\cap J=\emptyset}\cup
\{(\mathcal{R}_{J}\cap(I\setminus I_{ext}))\setminus J\}$,
such that 

$$
p_{J}(I_{i}\cap(I\setminus I_{ext}))=
\left\{
\begin{array}{cc}
(\mathcal{R}_{J}\cap(I\setminus I_{ext}))\setminus J, &
I_{i}\cap J\not=\emptyset, \\
I_{i}\cap(I\setminus I_{ext}), &
I_{i}\cap J=\emptyset.
\end{array}
\right.
$$
Then 

$$
\varGamma_{I_{ext}\sqcup_{f}\{I_{i}\}_{(I\setminus I_{ext})}}
\diagup\varGamma_{J_{ind}}=
\varGamma_{I_{ext}\sqcup_{p_{J}\circ f}
(\{I_{i}\}_{(I\setminus I_{ext})}\diagup\{I_{i}\}_{(J)})}.
$$
		
\item 
If a sub-diagram $\Gamma_{J}$
determined by an internal $\sigma-$invariant 
subset $J$ is not connected, then, according to
proposition 4.1 or proposition 4.7,
$J$ admits to the decomposition $J=\bigcup_{j=1}^{p}J_{j}$,
such that each $\varGamma_{J_{j}}$ is a 
connected component of $\Gamma_{J}$.
The quotient $\Gamma_{\{I_{i}\}}$ by $\Gamma_{J}$
should be

$$
\begin{array}{c}
\varGamma_{\{I_{i}\}}\diagup\varGamma_{J}=
(\sigma|_{I\setminus J},\{I_{i}\}\diagup(J_{j})) \\
=(\cdots((\varGamma_{\{I_{i}\}}\diagup\varGamma_{J_{1}}\diagup
\varGamma_{J_{2}})\cdots)\diagup \varGamma_{J_{p}}.
\end{array}
$$
Above quotient is also denoted by		
	
$$
\varGamma_{\{I_{i}\}}\diagup(\varGamma_{J_{j}})=
\Gamma_{\{I_{i}\}\diagup(J_{j})}.
$$
\end{itemize}
\end{remark}

The previous discussions show that the
quotient of Feynman diagrams can be reduced
to the situation of the ordinary graphs.
Then we have,

\begin{proposition}
Let $\varGamma_{\sigma,\{I_{i}\}}$ be a connected
Feynman diagram, $J\subset I=\mathcal{R}(\{I_{i}\})$ 
be an internal $\sigma-$invariant subset. Then,
$\varGamma_{\{I_{i}\}}\diagup\varGamma_{J}$ is
a connected Feynman diagram.
\end{proposition}

Furthermore, a conclusion similar to proposition 4.4
or theorem 2.1 is valid.

\begin{definition}
Let $(\sigma,\{I_{i}\}_{i=1}^{m})$, 
$(\lambda,\{J_{j}\}_{j=1}^{n})$
be two connected Feynman diagrams
satisfying $I\cap J=\emptyset$, where
$I=\mathcal{R}(\{I_{i}\}$,
$J=\mathcal{R}(\{J_{j}\}$. Then we define
the insertion of
$\varGamma_{\{J_{j}\}}$ into $\varGamma_{\{I_{i}\}}$ 
at $I_{a}$ by $\iota$ to be a pair 
$(\delta,\{I_{i}\}\circ_{I_{a}}^{\iota}
\{J_{j}\}_{(J\setminus J_{ext})})$ 
denoted by $\varGamma_{\{I_{i}\}}
\circ_{I_{a}}^{\iota}\varGamma_{\{J_{j}\}}$, or,
$\varGamma_{\{I_{i}\}\circ_{I_{a}}^{\iota}
\{J_{j}\}_{(J\setminus J_{ext})}}$ also, where
$1\leq a\leq m$, $\iota:I_{a}\to \{J_{j}\}$, 
$\delta|_{I}=\sigma$, $\delta|_{J\setminus J_{ext}}=\lambda$.
\end{definition}

\begin{remark}
\begin{itemize}
\item In the case of Feymman diagram, by definition 4.4
it is easy to check that 
$(\{I_{i}\}\circ_{I_{a}}^{\iota}\{J_{j}\}_{(J\setminus J_{ext})})_{(J\setminus J_{ext})}
=\{J_{j}\}_{(J\setminus J_{ext})}$, and 
$\delta|_{J\setminus J_{ext}}=\lambda|_{J\setminus J_{ext}}$.
$J\setminus J_{ext}$ is an internal 
$\delta-$invariant subset obviously. Then, by proposition 2.3
we have

$$
\varGamma_{(\{I_{i}\}\circ_{I_{a}}^{\iota}
\{J_{j}\}_{(J\setminus J_{ext})})\diagup
\{J_{j}\}_{(J\setminus J_{ext})}}=\Gamma_{\{I_{i}\}}.
$$
If $\Gamma_{\{I_{i}\}}$ and
$\Gamma_{\{J_{j}\}}$ are connected,
so is $\Gamma_{\{I_{i}\}\circ_{I_{a}}^{\iota}
\{J_{j}\}_{(J\setminus J_{ext})}}$.

\item Recalling Feymman diagram $(\sigma,\{I_{i}\})$
can be rewritten as 
$(\sigma,I_{ext}\sqcup_{f}\{I_{i}\}_{(I\setminus I_{ext})})$,
where $f$ satisfies $f^{-1}(I_{i}\setminus I_{ext})=I_{i}\cap I_{ext}$,
as an application of proposition 2.4, we know
that there is a map
$F:I_{ext}\to
\{I_{i}\}_{(I\setminus I_{ext})}\circ_{I_{a}}^{\iota^{\prime}}
\{J_{j}\}_{(J\setminus J_{ext})}$
($\iota^{\prime}=\iota|_{I_{a}\setminus I_{ext}}$)
such that

$$
(I_{ext}\sqcup_{f}\{I_{i}\}_{(I\setminus I_{ext})})
\circ_{I_{a}}^{\iota}\{J_{j}\}_{(J\setminus J_{ext})}=
I_{ext}\sqcup_{F}(\{I_{i}\}_{(I\setminus I_{ext})}
\circ_{I_{a}\setminus I_{ext}}^{\iota^{\prime}}
\{J_{j}\}_{(J\setminus J_{ext})}).
$$
\item In the same way as subsection 4.1,
we can define the composition of Feynman diagrams
as follows:

$$
\varGamma_{\{I_{i}\}}\circ\varGamma_{\{J_{j}\}}
=\sum\limits_{a,\iota}\varGamma_{\{I_{i}\}}
\circ_{I_{a}}^{\iota}\varGamma_{\{J_{j}\}}.
$$
The Lie bracket of Feynman diagrams
can be defined to be

$$
[\varGamma_{\{I_{i}\}},\varGamma_{\{J_{j}\}}]=
\varGamma_{\{I_{i}\}}\circ\varGamma_{\{J_{j}\}}-
\varGamma_{\{J_{j}\}}\circ\varGamma_{\{I_{i}\}}.
$$
Lie bracket above mentioned will be well defined
obviouly.	
\end{itemize}
\end{remark}

\subsection{Kontsevich graphs}

\begin{definition}
An admissible graph is a triple $(\sigma,\{I_{i}\},\{J_{j}\})$,
where $\{I_{i}\},\{J_{j}\}\in\mathcal{P}_{dis}^{2}(A)$,
$I\cap J=\emptyset$,
$\sigma$ is a map from $I\cup J$ to itself without 
fixed points, $\sigma^{2}=\sigma$,
where $I=\mathcal{R}(\{I_{i}\})$,
$J=\mathcal{R}(\{J_{j}\})$.
$\{I_{i}\}$ and $\{J_{j}\}$ satisfy

\begin{equation}
\begin{array}{c}
\sigma(I_{i})\cap I_{i}=\emptyset,\,\,
\sigma(J_{j})\cap J_{j^{\prime}}=\emptyset,\,\,
|\sigma(I_{i})\cap I_{i^{\prime}}|\leq 1,\,\,
|\sigma(I_{i})\cap J_{j}|\leq 1, \\
\forall\,i,\,i^{\prime},\,j,\,j^{\prime},\,i\not=i^{\prime}.
\end{array}
\end{equation}
We denote the
admissible graph $(\sigma,\{I_{i}\},\{J_{j}\})$
by $\varGamma_{\sigma,\{I_{i}\},\{J_{j}\}}$, or,
$\varGamma_{\{I_{i}\},\{J_{j}\}}$ for short.
$I\cup J$ is called
the total set of $\varGamma_{\{I_{i}\},\{J_{j}\}}$.
$\mathbf{Vert}(\varGamma_{\sigma,\{I_{i}\},\{J_{j}\}})
=\{I_{i}\}\cup\{J_{j}\}$. 
The vertices from $\{I_{i}\}$
are called the vetices of the first type,
the verices from $\{J_{j}\}$ are called 
the verices of the second type.
\end{definition}

\begin{remark} 
\begin{itemize}
\item The conditions in (4.5) means that:
\begin{itemize}
\item There no loop in $\varGamma_{\{I_{i}\},\{J_{j}\}}$.
\item For any two vertices there is at most one 
edge to connect them.
\item There is not edge to connect any two vertices of 
the second type.
\end{itemize}
\item In definition 4.8, we 
ignore the orientation of the edges. Now we discuss 
the issues of the orientation in details.
Recalling the previous discussions,
the edges of a graph are regarded as the
equivalent classes under the equivalent relation 
defined as above, each equivalent
class consists of two elements $(e,\sigma(e))$
and $(\sigma(e),e)$ ($e\in I\cup J$). 
To indicate the orientation of an edge,
we can only choose one element in each equivalent class,
for example, $(e,\sigma(e))$, to represent
an edge. If the first element $e\in I_{i}$, 
the second element $\sigma(e)\in I_{i^{\prime}}$
(or $\sigma(e)\in J_{j}$), then $I_{i}$ will be 
the starting point of the edge $(e,\sigma(e))$, 
$I_{i^{\prime}}$ (or $J_{j}$) will be the endpoint
of this edge. Furthermore, each $J_{j}$ can not be the
starting point of any edge. 

\item The connectivity of an admissible graph
can be defined in the same way as definition 4.1.
\end{itemize} 
\end{remark}

We now pay attention to subgraphs.

\begin{lemma}
Let $\varGamma_{\sigma,\{I_{i}\},\{J_{j}\}}$
be an admissible graph. Then
$\sigma(J)\cup J$ and $I\setminus\sigma(J)$
are two $\sigma-$invariant subsets. Additionally,
for a $\sigma-$invariant subset $K\subset I\cup J$,
we have

$$
K=(\sigma(K^{\prime})\cup K^{\prime})
\cup K^{\prime\prime},
$$
where $K^{\prime}=K\cap J$,
$K^{\prime\prime}=K\cap(I\setminus\sigma(J))$.
\end{lemma}

\begin{definition}
Let $(\sigma,\{I_{i}\},\{J_{j}\})$ be an admissible
graph, and $K\subset I\cup J$
be a $\sigma-$invariant subset such that
$K\cap J\not=\emptyset$, where
$J=\mathcal{R}(\{J_{j}\})$. 
A subgraph of $\varGamma_{\sigma,\{I_{i}\},\{J_{j}\}}$
determined by $K$ is a triple 
$(\sigma|_{K},\{I_{i}\}_{(K)},\{J_{j}\}_{(K)})$,
also denoted by $\varGamma_{\{I_{i}\}_{(K)},\{J_{j}\}_{(K)}}$,
or $\varGamma_{K}$ for short.
\end{definition}

Now we discuss the quotient and insertion of
the asmissible graphs. 

\begin{definition}
Let $\varGamma_{\sigma,\{I_{i}\},\{J_{j}\}}$
be a connected admissible graph, 
$K\subset I\cup J$ be a $\sigma-$invariant
subset satisfying:

\begin{itemize}
\item $|(\mathcal{R}_{\{J_{j}\},K\cap J}\setminus K)
\cap(\mathcal{R}_{\{I_{i}\},K\cap I}\setminus K)|\leq 1$,
\item $|(\mathcal{R}_{\{J_{j}\},K\cap J}\setminus K)
\cap I_{i}|\leq 1$, $I_{i}\cap K=\emptyset$,
\item $|(\mathcal{R}_{\{I_{i}\},K\cap I}\setminus K)
\cap J_{j}|\leq 1$, $J_{j}\cap K=\emptyset$,
\item $|(\mathcal{R}_{\{I_{i}\},K\cap I}\setminus K)
\cap I_{i}|\leq 1$, $I_{i}\cap K=\emptyset$.
\end{itemize} 
Then, we define the quotient of
$\varGamma_{\sigma,\{I_{i}\},\{J_{j}\}}$ by
$\varGamma_{K}$ to be a triple

$$
(\sigma|_{(I\cup J)\setminus K},
\{I_{i}\}\diagup\{I_{i}\}_{(K)},
\{J_{j}\}\diagup\{J_{j}\}_{(K)}),
$$
denoted by
$\varGamma_{\sigma,\{I_{i}\},\{J_{j}\}}\diagup
\varGamma_{K}$ also.
\end{definition}

\begin{definition}
Let $\varGamma_{\{I_{i}\},\{J_{j}\}}$ and
$\varGamma_{\{L_{l}\},\{K_{k}\}}$ be two connected
admissible graphs with structure maps
$\sigma$ and $\lambda$ 
($(I\cup J)\cap(L\cup K)=\emptyset$).
Let $1\leq a\leq |\{I_{i}\}|$,
$1\leq b\leq |\{J_{j}\}|$, $\iota:I_{a}\to\{L_{l}\}$,
$\kappa:J_{b}\to\{K_{k}\}$,
we define the insertion of
$\varGamma_{\{L_{l}\},\{K_{k}\}}$ into
$\varGamma_{\{I_{i}\},\{J_{j}\}}$ 
to be the triple

$$
(\delta,\{I_{i}\}\circ_{I_{a}}^{\iota}\{L_{l}\},
\{J_{j}\}\circ_{J_{b}}^{\kappa}\{K_{k}\}),
$$
where $\delta$ satisfies
$\delta|_{I\cup J}=\sigma$, $\delta|_{L\cup K}=\lambda$,
$I=\mathcal{R}(\{I_{i}\})$, $J=\mathcal{R}(\{J_{j}\})$,
$L=\mathcal{R}(\{L_{l}\})$,
$K=\mathcal{R}(\{K_{k}\})$.
The graph $(\delta,\{I_{i}\}\circ_{I_{a}}^{\iota}\{L_{l}\},
\{J_{j}\}\circ_{J_{b}}^{\kappa}\{K_{k}\})$ is also
denoted by $\varGamma_{\sigma,\{I_{i}\},\{J_{j}\}}
\circ_{I_{a},J_{b}}^{\iota,\kappa}
\varGamma_{\lambda,\{L_{l}\},\{K_{k}\}}$.
\end{definition}

Similar to the situations in
subsection 4.1 and 4.2, it is obvious
that the following conclusions are valid.

\begin{proposition}
\begin{itemize}
	\item Let admissible graph
	$\varGamma_{\sigma,\{I_{i}\},\{J_{j}\}}$
	and its subgraph $\varGamma_{K}$ are
	connected, where $K\subset(I\cup J)$
	is a $\sigma-$invariant subset. Then
	the quotient $\varGamma_{\sigma,\{I_{i}\},
	\{J_{j}\}}\diagup\varGamma_{K}$ is connected also.
	\item Let $\varGamma_{\{I_{i}\},\{J_{j}\}}$ and
	$\varGamma_{\{L_{l}\},\{K_{k}\}}$ be two connected
	admissible graphs. Then $(\delta,\{I_{i}\}
	\circ_{I_{a}}^{\iota}\{L_{l}\},
	\{J_{j}\}\circ_{J_{b}}^{\kappa}\{K_{k}\})$
	is a connected admissible graph. 
\end{itemize}
\end{proposition}

Let $(\sigma,\{I_{i}\},\{J_{j}\})$ and
$(\lambda,\{L_{l}\},\{K_{k}\})$ be two 
connected admissible graphs. Then,
it is natural for us to define the 
composition of the admissible graphs
to be

\begin{equation}
\begin{array}{c}
(\sigma,\{I_{i}\},\{J_{j}\})\circ
(\lambda,\{L_{l}\},\{K_{k}\}) \\
=\sum\limits_{a,b,\iota,\kappa}
(\delta,\{I_{i}\}\circ_{I_{a}}^{\iota}\{L_{l}\},
\{J_{j}\}\circ_{J_{b}}^{\kappa}\{K_{k}\}).
\end{array}
\end{equation}
Now we define the Lie bracket to be

\begin{equation}
\begin{array}{c}
[(\sigma,\{I_{i}\},\{J_{j}\}),
(\lambda,\{L_{l}\},\{K_{k}\})] \\
=(\sigma,\{I_{i}\},\{J_{j}\})\circ
(\lambda,\{L_{l}\},\{K_{k}\})-
(\lambda,\{L_{l}\},\{K_{k}\})\circ
(\sigma,\{I_{i}\},\{J_{j}\}).
\end{array}
\end{equation}
Observing the definition 4.11
and the formula (4.6),
the structure map $\delta$ on $(\delta,\{I_{i}\}
\circ_{I_{a}}^{\iota}\{L_{l}\},
\{J_{j}\}\circ_{J_{b}}^{\kappa}\{K_{k}\})$
is independent of the choices of $a,b,\iota,\kappa$,
even the order of the insertion,
thus, all admissible graphs concerning
(4.7) adapt same structure map.
Above fact implies that the Lie bracket (4.7) is
similar to the cartesian product of Lie algebras,
which means that the Lie bracket (4.7) will
satisfy the Jacobi identity.

\end{document}